\documentclass[12pt]{article}
\usepackage{amssymb,amsmath,amsthm,secdot}
\usepackage{bbm}

\usepackage[shortlabels]{enumitem}

\usepackage[mathscr]{euscript}
\usepackage{mathtools}

\makeatletter
\DeclareRobustCommand\widecheck[1]{{\mathpalette\@widecheck{#1}}}
\def\@widecheck#1#2{%
    \setbox\z@\hbox{\m@th$#1#2$}%
    \setbox\tw@\hbox{\m@th$#1%
       \widehat{%
          \vrule\@width\z@\@height\ht\z@
          \vrule\@height\z@\@width\wd\z@}$}%
    \dp\tw@-\ht\z@
    \@tempdima\ht\z@ \advance\@tempdima2\ht\tw@ \divide\@tempdima\thr@@
    \setbox\tw@\hbox{%
       \raise\@tempdima\hbox{\scalebox{1}[-1]{\lower\@tempdima\box
\tw@}}}%
    {\ooalign{\box\tw@ \cr \box\z@}}}
\makeatother

\sectiondot{subsection}
\usepackage{xcolor}
\usepackage{authblk}
\DeclareMathOperator{\I}{\mathbbm{1}}%
\DeclareMathOperator{\Law}{Law}%
\DeclareMathOperator{\sign}{sign}

\usepackage{color}

\newcounter{nameOfYourChoice}

\def\E{\hskip.15ex\mathsf{E}\hskip.10ex}
\def\P{\mathsf{P}}

\def\eps{\varepsilon}
\def\phi{\varphi}

\newtheoremstyle{Assump}%
  {3pt}
  {3pt}
  {\itshape}
  {}
  {\bfseries}
  {.}
  {.5em}
  {\thmname{#1} \thmnumber{#2} \thmnote{\normalfont#3}}

\newtheorem{Theorem} {Theorem}[section]
\newtheorem{Lemma}[Theorem]{Lemma}
\newtheorem{Proposition}[Theorem]{Proposition}

\theoremstyle{definition}\newtheorem{Example}[Theorem]{Example}
\theoremstyle{definition}\newtheorem{Remark}[Theorem]{Remark}
\theoremstyle{definition}\newtheorem{Definition}[Theorem]{Definition}

\theoremstyle{plain}\newtheorem{innercustomthm}{Assumption}
\newenvironment{Assumption}[1]
  {\renewcommand\theinnercustomthm{#1}\innercustomthm}
  {\endinnercustomthm}

\numberwithin{equation}{section}

\renewcommand{\ge}{\geqslant}
\renewcommand{\le}{\leqslant}

\renewcommand{\gg}[1]{\check{#1}}

\newcommand{\po}{\preceq}
\newcommand{\pl}{\po_{pos}}
\renewcommand{\pm}{\po_{st}}
\renewcommand{\ae}{\diamondsuit}

\newcommand{\nn}{\nonumber}
\newcommand{\wt}{\widetilde}

\newcommand{\B}{\mathcal{B}}
\newcommand{\C}{\mathcal{C}}

\newcommand{\cE}{\mathcal{E}}
\newcommand{\F}{\mathcal{F}}

\newcommand{\N}{\mathbb{N}}
\newcommand{\R}{\mathbb{R}}
\newcommand{\T}{\mathbb{T}}
\newcommand{\V}{\mathcal{V}}
\newcommand{\Z}{\mathbb{Z}}

\newcommand{\grad}{\nabla}

\title{Couplings via comparison principle and exponential ergodicity of SPDEs in the hypoelliptic setting}

\date{\today}

\author[$*$,1,2]{Oleg Butkovsky}
 \author[$**$,1]{Michael Scheutzow}
\affil[1]{\small {Technische Universit\"at Berlin,

Institut f\"ur Mathematik, MA 7-5, Fakult\"at II,

Strasse des 17.~Juni 136, 10623 Berlin, Germany. \bigskip
}}

\affil[2]{\small {Weierstrass institute,

Mohrenstraße 39,

10117 Berlin, Germany. \bigskip
}}

\begin{document}

\maketitle
\renewcommand{\thefootnote}{*}
\footnotetext{Email: \texttt{oleg.butkovskiy@gmail.com}. Supported in part by DFG Research Unit FOR 2402 and European Research Council grant 683164.}

\renewcommand{\thefootnote}{**}
\footnotetext{Email: \texttt{ms@math.tu-berlin.de}. Supported in part by DFG Research Unit FOR 2402.}

\begin{abstract}
We develop a general framework for studying ergodicity of order--preserving Markov semigroups. We establish natural and in a certain sense optimal  conditions for existence and uniqueness of the invariant measure and exponential convergence of transition probabilities of an order--preserving Markov process. As an application, we show exponential ergodicity and exponentially fast synchronization--by--noise of the stochastic reaction--diffusion equation in the hypoelliptic setting. This refines and complements  corresponding results of Hairer, Mattingly (2011).
\end{abstract}

\section{Introduction}

The goal of this article is to build a framework for analyzing ergodic properties of order--preserving Markov processes and to provide simple, verifiable, yet general enough sufficient conditions for exponential ergodicity. This framework turns out to be especially powerful for investigating ergodicity of order--preserving stochastic PDEs with highly degenerate additive forcing. Our main example is the stochastic reaction--diffusion equation in the hypoelliptic setting. We show that even if noise comes to the system only from one Brownian motion, then (under certain conditions) this SPDE has a unique invariant measure and its transition probabilities converges to it exponentially in the Wasserstein metric. We also establish exponentially fast synchronization--by-noise of the solutions to this equation. This refines \cite[Remark 8.22]{HM11} and complements \cite[Theorem 8.21]{HM11}.

In the mathematical physics literature there is a growing interest in ergodic behavior of nonlinear PDEs forced by smooth in space noise acting only on a few Fourier modes, see, e.g., \cite{Mat99,Mat2002, HM06, CGHV14, KS12}. Since the noise is smooth in space, it is usually relatively easy to show
that these SPDEs have a unique solution and that this solution is a Markov process. On the other hand,
since the solution at any fixed time is an infinite--dimensional random variable and the noise acts only onto finitely many degrees of freedom, these processes do not get “enough” noise and
hence they are typically only Feller but not strong Feller, see also the discussion in \cite[Section 9]{Hai08}. This makes analyzing their ergodic behavior much more challenging.

Indeed, recall that ergodicity of strong Feller processes can be established using the standard classical approach, which combines a local mixing condition on a certain set (the small set condition) and a recurrence condition, see, e.g., \cite{MT}. Unfortunately, this method is usually not applicable for Markov processes which are only Feller and not strong Feller because they do not have good mixing properties, see also a detailed discussion in \cite[Section~1]{HMS11}). To study ergodicity of these processes three alternative strategies have been suggested recently.

The first approach  was developed in \cite{HM06, HM08, HM11}, see also \cite[Section 11]{Hai08}. It introduces the asymptotic strong Feller (ASF) property, which serves as a replacement for the strong Feller property. It is shown there that if a Markov process satisfies ASF as well as certain recurrence and topological irreducibility conditions then the process is exponentially ergodic. Note that for many Markov processes verifying ASF might be rather challenging. In particular, while this method works  quite well for stochastic Navier--Stokes (SNS) equations on a torus in the hypoelliptic setting, it is not so clear how to check ASF for the SNS equation on a bounded domain, see the discussion in  \cite[Section~1]{GMR17}.

Another approach establishes exponential ergodicity using generalized couplings \cite{Mat2002, H02,HMS11,GMR17, KS18,BKS}. Recall that a coupling is a pair of stochastic processes with
given marginal distributions. By contrast, a generalized coupling is a pair of stochastic processes, whose marginals are not necessarily equal to a prescribed pair of probability distributions, but are in a sense close to this pair. Clearly, constructing a generalized coupling is much easier than constructing a coupling. Furthermore, it is shown in the papers mentioned above that existence of a generalized coupling with certain nice properties yields exponential ergodicity. This approach works quite well for a large class of SPDEs in the effectively elliptic setting (that is, when noise acts in a finite but large enough number of directions), but is less useful for studying SPDEs in the hypoelliptic setting.

Finally, the third main approach was introduced in \cite{HMS11}. It utilizes the notion of a $d$--small set (a generalization of a small set), which is particularly well adapted to the study of Markov processes with bad mixing properties. This approach provides another set of sufficient conditions for exponential ergodicity, and it works quite well with stochastic delay equations and nonlinear autoregressions. Unfortunately, verifying this set of conditions for SPDEs is rather difficult.

Our new approach developed in this paper is somehow orthogonal to all of the strategies mentioned above. It is specifically targeted at order preserving Markov semigroups, that is, the semigroups which map increasing bounded functions to increasing bounded functions. On the one hand, this significantly reduces the applicability of this approach; for example, it cannot be used to study stochastic Navier--Stokes equations. On the other hand, for order--preserving Markov processes (e.g., stochastic reaction--diffusion equations) it allows to obtain exponential ergodicity under very weak assumptions; this is rather  difficult (or maybe impossible) to achieve with other methods.

The main result (Theorems~\ref{T:main} and \ref{T:maininv}) is quite general. It shows that if an order--preserving Markov semigroup satisfies additionally a swap condition (i.e., two Markov processes started with initial conditions $x$, $y$ with $x\po y$ can change their order by time $1$ with a small but positive probability), then under a standard Lyapunov--type condition as well as a certain technical assumption the process is exponentially ergodic. We also show that this swap condition cannot be omitted.

We apply the obtained theorems to establish exponential ergodicity of stochastic reaction--diffusion equations on
a $d$--dimensional torus $\T^d$, $d\in\N$
\begin{equation}\label{srde}
d u(t,\xi)=[\Delta u(t,\xi)+f(u(t,\xi),\xi)]dt+\sum_{k=1}^m \sigma_k(\xi)dW^k(t), \quad \xi\in \T^d,\, t\ge0,
\end{equation}
where $(W^1, W^2, \hdots, W^m)$ are independent standard Brownian motions; $f$, $\sigma_k$ are continuous functions acting from $\R\times\T^d\to\R$ and $\T^d\to\R$, respectively and satisfying certain conditions. It is clear that if $m=0$ (no noise), then this equation might have multiple invariant measures.
On the other hand, if $m=\infty$ (noise acts in every direction), then, under certain additional assumptions on $\sigma$, the process $u$ is strong Feller and it can be shown by the classical methods that it has a unique invariant measure \cite[Sections~7 and 11]{DPZ96}. Thus, it is natural to ask what the smallest number of directions $m$ that have to be perturbed by noise is, so that equation \eqref{srde} still has a unique invariant measure.

Using the ASF method described above, it was shown in \cite[Remark 8.22]{HM11} that if $f$ is a polynomial and $m=3$, then equation \eqref{srde} has a unique invariant measure and is exponentially ergodic. We extend this result and show exponential ergodicity of $u$ already if $m=1$ (that is, when noise acts only in one direction), see Theorem~\ref{T:invmeasure} and Remark~\ref{R:invmeasure}; we also do not rely on a specific form of $f$. Note that the convergence to the invariant measure is established in the Wasserstein metric; Theorem~\ref{T:weareoptimal} shows that in the case $m<\infty$ this is optimal. Namely, if no additional assumptions are imposed, then, contrary to the case $m=\infty$, the transition probabilities of Markov process \eqref{srde} might \textbf{not} converge to the invariant measure in the total variation metric. Finally, in Theorem~\ref{T:sync}, we show that any two solutions to \eqref{srde} launched with the same noise from different initial conditions converge to each other exponentially fast (synchronization by noise).

The idea that order--preservation helps to obtain better convergence rates of a Markov process is not new; it was used to study interacting particle systems since 1970s, see \cite{Li}. However, it is not clear at all how to extend the techniques used in this book beyond the framework of interacting particle models.
Order--preserving Markov processes on a general state spaces were considered in \cite{LT, RT}. However, the methods developed there rely on the small set condition. Since in the current paper we study processes on a general state space with bad mixing properties, where this condition might not hold, unfortunately the ideas of \cite{Li, LT, RT}  cannot be applied in our case.

It is interesting to compare our results with \cite{FGS}. In that paper the authors consider an order--preserving random dynamical system (RDS) with two additional properties: it has a unique invariant measure and it weakly converges to this measure. It is shown there that this implies that  any two trajectories of the RDS converge to each other in probability. By contrast, in the current paper we start with an order--preserving Markov process and prove uniqueness of an invariant measure and convergence of transition probabilities.

Our main tool is a new version of the coupling method specifically tailored for order--preserving Markov processes, see the proof of Theorem~\ref{T:main}. This is combined with an analysis of the relations between stochastic domination and expected distance of random variables, see Section~\ref{S:OPanddistance}. There we continue the study initiated in \cite[Proposition~1]{CS} and \cite[Proposition~2.4]{FGS}. Note however, that the methods introduced in \cite{CS} and \cite{FGS} cannot be used to get a quantitative bound even in the case where the Markov process has  state space $\R$, see Example~\ref{E:32}. Therefore we apply a new technique.

While we study in detail only the stochastic reaction--diffusion equations on the torus, the strategy developed in this paper should also work in a very similar way for other order--preserving SPDEs   including stochastic reaction--diffusion equation on a bounded domain and stochastic porous medium equations. We would like  to mention also  that after the  preprint of our paper became available, our technique was extended and adapted to study regularization by noise for singular SPDEs \cite{GT}.

The rest of the paper is organized as follows. We present our main results in Section~\ref{S:2}. In Section~\ref{S:OPanddistance} we investigate the relations between stochastic domination and average distance of random variables. Section~\ref{S:3} is devoted to a detailed study of ergodicity of stochastic reaction--diffusion equations. The proofs of the main results are placed in Section~\ref{S:Proofs}.

\textbf{Convention on constants}. Throughout the paper $C$ denotes a positive constant whose value may change from line to line.

\bigskip

\noindent \textbf{Acknowledgments}. The authors are  deeply indebted to Benjamin Gess and Konstantinos Dareiotis for their help, patience and detailed explanations of some parts of the theory of parabolic PDEs. We also would like to thank M\'at\'e Gerencs\'er, Alessandra Lunardi, Lenya Ryzhik, Vladim\'{\i}r \v{S}ver\'{a}k, and Pavlos Tsatsoulis for useful comments and helpful discussions. We are also grateful to two anonymous reviewers for providing helpful comments on an earlier draft of the manuscript. Part of the work on the project has been done during the visit of OB to the Institute of Science and Technology --- Austria (IST) and Max-Planck-Institut für Mathematik in den Naturwissenschaften (Leipzig, Germany). OB is very grateful to these institutions for providing excellent working conditions.  OB has received funding from the European Research Council (ERC) under the
European Union’s Horizon 2020 research and innovation program (grant agreement No. 683164).

\section{A general framework for ergodicity for order--\! preserving Markov processes}\label{S:2}

We begin by introducing the basic notation. Let $(E,\rho)$ be a Polish space with partial order $\po$ such that the set
\begin{equation}\label{closedset}
\Gamma:=\{(x,y)\in E\times E\colon x\po y\}
\end{equation}
is closed. Let $\cE=\B(E)$ be the Borel $\sigma$-field. For sets $A,B\in\cE$ we will write $A\po B$ if for any $a\in A$, $b\in B$ we have $a\po b$.  Denote by $\mathcal{P}(E)$ the set of all probability measures on $(E,\cE)$. 
Let $\{P_t\}_{t\ge0}=\{P_t(x,A),x\in E, A\in\mathcal E\}_{t\ge0}$  be a Markov transition function over $E$  and denote by $\{\P_x, x\in E\}$ the corresponding Markov family; that is $\P_x$ is the law of the Markov process $\{X_t, t\ge0\}$ with the given transition function and initial condition $X_0=x$. The law of $X$ will be understood in the sense of finite--dimensional distributions;  that is, we will not rely on the trajectory--wise properties of $X$.

For a measurable function $r:E\times E\to[0,1]$, we consider the corresponding coupling distance $W_r:\mathcal{P}(E)\times\mathcal{P}(E)\to\R_+$ given by
\begin{equation*}
W_r(\mu,\nu):=\inf_{\lambda\in\C(\mu,\nu)}\int_{E\times E} r(x,y) \lambda(dx,dy),\quad \mu,\nu\in\mathcal{P}(E),
\end{equation*}
where  $\C(\mu,\nu)$ is  the set of all couplings between $\mu$ and $\nu$, i.e.,  probability measures on $(E\times E,\mathcal{E}\otimes\mathcal{E})$ with marginals $\mu$ and $\nu$. If $r$ is a lower semicontinuous metric on $E$, then $W_r$ is the usual Kantorovich--Wasserstein distance. If $r$ is the discrete metric, i.e., $r(x, y) =\I(x\neq y)$, then $W_r$ is the total variation distance, which will be denoted further by $d_{TV}$.

Let us now recall the standard definitions related to the partial order $\po$; we refer to, e.g., \cite[Chapter IV]{L92} for a detailed discussion.
\begin{Definition}
\begin{enumerate}[{\rm(i)}]
\item A function $f\colon E\to\R$ is called \textit{increasing} if for any $x,y\in E$ with $x\po y$ we have $f(x)\le f(y)$.
\item Let $\mu,\nu\in\mathcal{P}(E)$ be two probability measures. We say that $\nu$ \textit{stochastically dominates}~$\mu$ and denote it by $\mu\pm\nu$ if for any bounded measurable increasing function $f\colon E\to \R$ we have
$$
\int_E fd\mu\le \int_E fd\nu.
$$
\item We say that  a Markov transition function $\{P_t\}_{t\ge0}$ is \textit{order preserving} if for any $t>0$ and $x,y\in E$ such that $x\po y$ we have
$$
P_t(x,\cdot)\pm P_t(y,\cdot).
$$
\end{enumerate}
\end{Definition}

In other words,  a Markov transition function is order preserving if it maps bounded increasing functions to bounded increasing functions. Examples of Markov processes with an order preserving transition function include stochastic--reaction diffusion equations, stochastic porous media equations and others, see, e.g., \cite{FGS}.

\begin{Remark}
Strassen's theorem (see, e.g., \cite[Theorem IV.2.4]{L92}) provides the following coupling definition of stochastic domination, which is equivalent to the one stated above. We have $\mu\pm\nu$ if and only if there exist random elements $X,Y\colon\Omega\to E$ such that $\Law(X)=\mu$, $\Law(Y)=\nu$ and $X\po Y$.
\end{Remark}

Now we are ready to present our main results.

\begin{Theorem}\label{T:main}
Suppose that there exist measurable functions $V\colon E\to[0,\infty)$, $\phi\colon E\to\R$, $d\colon E\times E\to [0,\infty)$, such that the  following conditions hold:
\begin{enumerate}[label={\rm(\arabic*)}]
 \item\label{Cond:pervij} the Markov transition function $(P_t)_{t\ge0}$ is order--preserving;
\item\label{Cond:vtoroj} the function $V$ is a Lyapunov function, that is, there exist constants  $\gamma, K>0$ such that
 \begin{equation}\label{contlyap}
 P_t V(x)\le V(x)-\gamma\int_0^t P_s V(x)\,ds+Kt,\quad t\ge0,\, x\in E;
  \end{equation}
 \item\label{Cond:tretij} if $x,y\in E$ and  $x\po y$, then $0\le d(x,y)\le \phi(y)-\phi(x)$;
 \item \label{Cond:chetv} for any $x\in E$ we have $M(x):=\sup_{t\ge0} P_t \phi^2(x)<\infty$;
 \item\label{Cond:pyat} there exist sets $A, B\in\cE$ and $\eps>0$ such that $A\po B$ and for any $x\in \{V\le 4K/\gamma\}$ we have
 \begin{equation}\label{nonmixing}
P_1(x,A)>\eps \text{ and } P_1(x,B)>\eps.
 \end{equation}
 \setcounter{nameOfYourChoice}{\value{enumi}}
 \end{enumerate}

 Then for any $\theta>0$ there exist constants $C,\lambda>0$ such that for any $x,y\in E$
 \begin{equation}\label{finalres}
W_{d\wedge1}(P_t(x,\cdot),  P_t(y,\cdot))\le C (1+V(x)+V(y))^{\theta}(1+M(x))^{\theta} \exp(-\lambda t),\quad t\ge0.
 \end{equation}
\end{Theorem}

\begin{Theorem}\label{T:maininv}
Suppose that all conditions of Theorem~\ref{T:main} are satisfied. Assume further that
\begin{enumerate}[label={\rm(\arabic*)}]
\setcounter{enumi}{\value{nameOfYourChoice}}
 \item \label{Cond:dist}  there exists $\delta\in(0,1]$ such that $\rho\le d^{\delta}$, where $\rho$ is the metric on the Polish space $E$.
 \item \label{Cond:MV} there exists $K>0$, $\kappa>0$ such that $M(x)^\kappa\le K+K V(x)$ for all $x\in E$.
\end{enumerate}

 Then  the Markov semigroup has a unique invariant measure $\pi$. Further, for any $\theta>0$ there exist constants $C,\lambda>0$ such that for any $x\in E$
 \begin{equation}\label{finalres2}
W_{\rho\wedge1}(P_t(x,\cdot), \pi)\le C (1+V(x))^{\theta} \exp(-\lambda t),\quad t\ge0.
 \end{equation}
\end{Theorem}

\begin{proof}[Sketch of the proof of Theorems~\ref{T:main} and \ref{T:maininv}]
Here, for the convenience of the reader, we provide just a very brief roadmap of the proof; a complete proof is given in Section~\ref{S:Proofs}.

Fix $x,y\in E$, $t>0$. The proof splits into two independent parts. First, we use a new version of the coupling method and utilize conditions \ref{Cond:pervij}, \ref{Cond:vtoroj}, \ref{Cond:pyat} to construct random elements $Z^x, Z^y, \wt Z^x$ taking values in $E$ with the following properties:
\begin{align}\label{boundt23}
&\Law (Z^x)=\Law (\wt Z^x)=P_t(x,\cdot),\quad \Law (Z^y)=P_t(y,\cdot);\nn\\
&\P(Z^x\po Z^y \po \wt Z^x)\ge 1-C_1(1+V(x)+V(y))e^{-C_2 t},
\end{align}
for some universal constants $C_1, C_2>0$. Second, using the ideas developed in Section~\ref{S:OPanddistance}, we transform the bound \eqref{boundt23} into the following inequality:
\begin{equation}\label{boundt223}
\E [d(Z^x,Z^y)\wedge1]\le 1-C_3(1+V(x)+V(y))(1+M(x))e^{-C_4 t},
\end{equation}
where $C_3, C_4>0$ are again some universal constants. It is at this step, where we are using conditions \ref{Cond:tretij} and \ref{Cond:chetv}. Since
$\Law (Z^x)=P_t(x,\cdot)$ and  $\Law (Z^y)=P_t(y,\cdot)$, inequality \eqref{boundt223} yields \eqref{finalres} and \eqref{finalres2}.
\end{proof}

Theorems~\ref{T:main} and \ref{T:maininv} provide general sufficient conditions for an order--preserving Markov process to be ergodic. Condition \ref{Cond:tretij} of Theorem~\ref{T:main} is actually a condition on the space $(E,\po)$ and $d$ rather than on the Markov semigroup. As shown below it is satisfied in many natural situations, including $E=\R$, $E=L_p$, $p\in[1,\infty)$, $E=\mathscr{B}^{\alpha}_{p,\infty}$, $\alpha<0$, $p\in[1,\infty)$; the latter stands for the Besov space of regularity $\alpha$ and integrability $p$, see, e.g., \cite[Section~2.7 and Proposition 2.93]{BCD}. Thus, the only additional assumption for exponential ergodicity (apart from the standard Lyapunov and moment-type conditions) is the swap condition \ref{Cond:pyat}. It tells that the state space $E$ contains two sets, one preceding the other, and locally uniformly in the initial condition the Markov process has a small chance to be in either of these sets. The following simple example explains why this condition cannot be dropped.

\begin{Example}
Introduce the following trivial order on $E$: for $x,y\in E$ we have $x\po y$ if and only if $x=y$. It is clear, that for this order any Markov semigroup is order--preserving. We also see that the set $\Gamma$ defined in \eqref{closedset} is closed; furthermore, conditions \ref{Cond:tretij} and \ref{Cond:chetv} of Theorem~\ref{T:main} trivially hold with $\phi\equiv0$. Thus, any Markov process on $E$, that has a Lyapunov function satisfies conditions  \ref{Cond:pervij}--\ref{Cond:chetv} of Theorem~\ref{T:main}. It is well--known that this is not enough for uniqueness of the invariant measure. Thus, the swap condition \ref{Cond:pyat} cannot be omitted.
\end{Example}

We also would like to emphasize that the swap condition \ref{Cond:pyat} is very different in nature from the small set condition or other minorization--type conditions, which are imposed within the classical framework, see e.g., \cite{Ros}. Indeed, a minorization condition guarantees good mixing properties of transition kernels and, in particular, bounds on the total variation distance between the kernels. On the other hand,
the swap condition does not yield such bounds since nothing was assumed about mixing on the sets $A$ and $B$.
Lemma~\ref{L:Michael} shows how the swap condition can be verified for the stochastic--reaction diffusion equation.

Now let us provide natural examples of spaces $E$ for which condition \ref{Cond:tretij} of Theorem~\ref{T:main} holds.

\begin{Example}
Let $S$ be a countable set. Put $E=\{0,1\}^S$ equipped with the distance $d(x,y):=\sum_{i=1}^\infty 2^{-i}|x_i-y_i|$, where $x,y\in\{0,1\}^S$ and $x=(x_1,x_2,\dots)$, $y=(y_1,y_2,\dots)$. This space often appears within the context of interacting particle systems. Consider the following partial order: if $x,y\in\{0,1\}^S$, then
\begin{equation*}
x\po_{inc} y \text{ if and only if } x_i\le y_i  \text{ for all $i\in\N$}.
\end{equation*}
Then condition \ref{Cond:tretij} holds for the function $\phi(x):=\sum_{i=1}^\infty 2^{-i}x_i$, $x\in\{0,1\}^S$.
\end{Example}

\begin{Example}
Put $E=\R$ equipped with the standard distance, $d(x,y):=|x-y|$, $x,y\in\R$, and consider the standard order $\le$. Then condition \ref{Cond:tretij} holds for the function $\phi(x):=x$, $x\in\R$.
\end{Example}

\begin{Example}\label{E:main}
Put $E:=L_p(D,\R)$, where $p\ge1$ and an arbitrary domain $D\subset \R^n$, $n\in\N$. Let $\|\cdot\|_{L_p}$ be the standard $L_p$ norm in this space. Consider the following partial order
\begin{equation}\label{aaa}
x\pl y \text{ if and only if } x(\xi)\le y(\xi) \text{ for almost all $\xi\in D$}
\end{equation}
and let
$d(x,y):=\|x-y\|_{L_p(D)}^p$, $x,y\in L_p(D,\R)$. Then there exists a function $\phi$ such that the partial order $\pl$ and function $d$ introduced above satisfy condition \ref{Cond:tretij} of  Theorem~\ref{T:main}. Furthermore, the set $\Gamma$ defined in \eqref{closedset} is closed. We postpone the proof of this statement to Section~\ref{S:Proofs}.
\end{Example}

\begin{Example}[\cite{GT}]
This example is due to \cite{GT}. Put $E:=\mathscr{B}^{\alpha}_{p,\infty}(D,\R)$, where $p\ge1$, $\alpha<0$ and an arbitrary domain $D\subset \R^n$, $n\in\N$.
Let $\|\cdot\|_{\mathscr{B}^{\alpha}_{p,\infty}}$ be the standard Besov norm in this space, see, e.g.,  \cite[Definition~2.68]{BCD}. Consider the following partial order
\begin{equation*}
x\po_{Besov} y \text{ if and only if } \langle x,\phi\rangle\pl \langle y,\phi\rangle \text{ for any non-negative $\phi\in\C^\infty$}
\end{equation*}
and let
$d(x,y):=\|x-y\|_{\mathscr{B}^{\alpha}_{p,\infty}}^p$, $x,y\in \mathscr{B}^{\alpha}_{p,\infty}(D,\R)$. Then, as shown in \cite[Lemma~A.1]{GT}, there exists a function $\phi$ such that the partial order $\po_{Besov}$ and function $d$  satisfy condition \ref{Cond:tretij} of  Theorem~\ref{T:main}.
\end{Example}

\section{Stochastic domination and distance between random variables}\label{S:OPanddistance}

In this section we explore the connections between stochastic domination and expected distance of random variables, thus continuing the analysis initiated in \cite{FGS}. Recall that we are given a Polish space $E$ with metric $\rho$. The following statement played a key role in establishing synchronization--by--noise results in \cite{FGS}.

\begin{Proposition}[{\cite[Proposition~2.4]{FGS}}] Let $X_t$, $Y_t$ be two stochastic processes taking values in $E$ such that $X_t(\omega)\po Y_t(\omega)$ for all $t\ge0$, $\omega\in\Omega$. Suppose that  $X_t$ and $Y_t$ converge weakly as $t\to\infty$ to a random element with the law $\mu$ . Then,
\begin{equation}\label{quantbound}
\E[\rho(X_t, Y_t)\wedge1]\to0\quad \text{as $t\to\infty$}.
\end{equation}
\end{Proposition}

It is natural to ask whether the above statement can be quantified. More precisely, assume  additionally that
$$
d_{TV}(\Law(X_t),\mu)\le r(t),\quad d_{TV}(\Law(Y_t),\mu)\le r(t),
$$
for some rate function $r$ going to $0$ as $t\to\infty$. One can ask whether these bounds guarantee a quantitative estimate on $\E[\rho(X_t, Y_t)\wedge1]$. Quite surprisingly, the answer to this question is negative: without any additional assumptions $\E[\rho(X_t, Y_t)\wedge1]$ might tend to $0$ very slowly even when $r(t)$ converges exponentially fast. This is illustrated by the following example.

\begin{Example}\label{E:32}
Let $E=\R$ be equipped with the standard order $\le$. Let $(p_i)_{i\in\Z_+}$ be a sequence of positive numbers summing up to $1$. Let $X$ be a random variable which with probability $p_i$ is uniformly distributed on $[2^i,2^{i+1}]$, $i\in\Z_+$.
Define for $n\in\Z_+$
$$
Y_n:=X\I(X\notin[2^n,2^{n+1}])+(X+1)\I(X\in[2^n,2^{n+1}]);\quad X_n=X.
$$
Then for each $n$ we clearly have $X_n\le Y_n$ and it is immediate to see that for $\mu:=\Law(X)$ one has
\begin{equation*}
d_{TV}(\Law(X_n),\mu)=0;\quad d_{TV}(\Law(Y_n),\mu)\le p_n 2^{-n}.
\end{equation*}
On the other hand,
$$
\E[|X_n-Y_n|\wedge1]=p_n,
$$
which is much slower than $p_n 2^{-n}$.
\end{Example}

\medskip
Thus, to quantify bounds in \eqref{quantbound} we need to impose extra assumptions. This is where condition~\ref{Cond:tretij} of Theorem~\ref{T:main} pops up.
\begin{Lemma}\label{L:33}
Assume that condition \ref{Cond:tretij} of Theorem~\ref{T:main} is satisfied for measurable functions $\phi\colon E\to\R$, $d\colon E\times E\to[0,1]$. Suppose further that there exist a function $\psi\colon E\to\R_+$ and $k\in(0,1)$ such that
\begin{equation}\label{philipbound}
|\phi(x)-\phi(y)|\le d(x,y)^k(\psi(x)+\psi(y)),\quad x,y\in E.
\end{equation}
Let $X, Y$ be random elements $\Omega\to E$ such that $X\po Y$ a.s., $\E|\phi(X)|<\infty$, $\E|\phi(Y)|<\infty$ and $W_d(\Law(X),\Law(Y))\le \eps$ for some $\eps>0$.  Then we have
\begin{equation}\label{mbound33}
\E d(X,Y)\le \eps^{k}\bigl((\E[\psi(X)^{1/(1-k)}])^{1-k}+(\E[\psi(Y)^{1/(1-k)}])^{1-k}\bigr).
\end{equation}
\end{Lemma}
\begin{proof}
We begin by observing that since $X\po Y$ a.s., condition \ref{Cond:tretij} of Theorem~\ref{T:main} yields
\begin{equation}\label{firststepsynch}
\E d(X,Y)\le \E\phi(Y)-\E\phi(X).
\end{equation}
Fix $K>1$ and let $\wt X$, $\wt Y$ be random elements such that $\Law(\wt X)=\Law(X)$, $\Law(\wt Y)=\Law(Y)$ and $\E d(\wt X,\wt Y)\le K\eps$ (they exist since $W_d(\Law(X),\Law(Y))\le \eps$). Then we continue \eqref{firststepsynch} as follows, using the fact that $d$ is bounded by 1:
\begin{align*}
\E d(X,Y)&\le \E\phi(\wt Y)-\E\phi(\wt X)\\
&\le \E [d(\wt X,\wt Y)^k(\psi(\wt X)+\psi(\wt Y))]\\
&\le K^{k}\eps^{k}\bigl((\E[\psi(\wt X)^{1/(1-k)}])^{1-k}+(\E[\psi(\wt Y)^{1/(1-k)}])^{1-k}\bigr)\\
&=K^{k}\eps^{k}\bigl((\E[\psi(X)^{1/(1-k)}])^{1-k}+(\E[\psi(Y)^{1/(1-k)}])^{1-k}\bigr).
\end{align*}
Since $K>1$ was arbitrary, this yields the statement of the lemma.
\end{proof}

The lemma above will be very useful for obtaining exponential bounds on the synchronization\! --by--noise. On the other hand, to complete the second part of the proof of Theorem~\ref{T:main} (see its sketch above in Section~\ref{S:2}), we need to solve the opposite problem. More precisely, Lemma~\ref{L:33} considers the case when one random element is less than the other everywhere but their laws are different. The following lemma studies the situation where the laws of the random elements coincide, but one is less than the other with probability smaller than $1$.

\begin{Lemma}\label{L:44}
Assume that condition \ref{Cond:tretij} of Theorem~\ref{T:main} is satisfied for measurable functions $\phi\colon E\to\R$, $d\colon E\times E\to[0,1]$.
Let $X, Y, \wt X$ be random elements $\Omega\to E$ such that $\Law(X)=\Law(\wt X)$ and
$$
\P(X\po Y\po \wt X)\ge1-\eps
$$
for some $\eps>0$. Then for any $p,q>1$ such that $1/p+1/q=1$ we have
$$
\E d(X,Y)\le 2 \eps^{1/p}(\E |\phi(X)|^q)^{1/q}+\eps.
$$
\end{Lemma}

Again, we see that the statement of the lemma is satisfied only under some extra conditions (one needs condition \ref{Cond:tretij} of Theorem~\ref{T:main} and a reasonable bound on $\E |\phi(X)|^q$). The following two examples show that these extra conditions cannot be dropped.

\begin{Example}
Let $E=\C([0,1],\R)$ be equipped with the sup norm $\|\cdot\|$ and order $\pl$ (defined as in \eqref{aaa}). For $x,y\in E$ put $d(x,y):=\|x-y\|\wedge1$. Let $n\in\N$. For $k\in\{1,2,\hdots,n\}$ let
$f_k$ be an element of $\C([0,1],\R)$ taking values in $[0,1]$ such that
$$
f_k(\xi)=1\,\,\,\text{for $\xi\in[0,k/n]$};\quad f_k(\xi)=0\,\,\,\text{for $\xi\in[(k+1)/n,1]$}.
$$
Then it is easy to see that $f_k\pl f_l$ whenever $k\le l\le n$. Let $\eta$ be uniformly distributed on the set $\{1,2,\hdots,n\}$. Put $X:=f_\eta$ and $Y=\wt X:=f_{\eta+1}$, where the summation is taken $\mod n$. Then $X$ and $\wt X$ are bounded, $\Law(X)=\Law(\wt X)$, $\P(X\po Y\po \wt X)=1-1/n$, but $\E [\|X-Y\|\wedge1]=1$.

It is easy to see that in this example condition \ref{Cond:tretij} of Theorem~\ref{T:main} does not hold. Indeed, if a function $\phi$ satisfies this condition, then for any $n\in\N$,  $k\in\{1,2,\hdots,n-1\}$ we have $\phi(f_{k+1})\ge\phi(f_{k})+d(f_{k},f_{k+1})=\phi(f_{k})+1$. Therefore
\begin{equation}\label{errorphi}
\phi(\gg1)=\phi(f_n)\ge\phi(f_1)+n-1\ge\phi(\gg0)+n,
\end{equation}
where $\gg1$ and $\gg0$  denote elements of $E$ identically equal to $1$ and $0$, respectively. Since $n$ was arbitrary, we see that \eqref{errorphi} implies that $\phi(\gg1)$ cannot be finite, which is impossible.
\end{Example}

\begin{Example}
Let $E=\R$ be equipped with the standard order $\le$. For $x,y\in E$ put $d(x,y):=|x-y|\wedge1$. Let $n\in\N$ and let $X$ be uniformly distributed on the set $\{1,2,\hdots,n\}$. Let $Y=\wt X=(X+1)\I(X<n)+\I(X=n)$. Then condition \ref{Cond:tretij} of Theorem~\ref{T:main} holds with $\phi(x)=x$, $\Law(X)=\Law(\wt X)$, $\P(X\po Y\po\wt X)=1-1/n$, but $\E [|X-Y|\wedge1]=1$.

Thus, we see that in this example condition \ref{Cond:tretij} of Theorem~\ref{T:main} holds, but $\E |\phi(X)|=\E X=n/2+1/2$ can be arbitrarily large for large $n$.
\end{Example}

\begin{proof}[Proof of Lemma~\ref{L:44}]
We begin by observing that the function $\phi$ is increasing, that is $x\po y$ implies $\phi(x)\le \phi(y)$ for any $x,y\in E$. Fix $p,q>1$ such that $1/p+1/q=1$. Without loss of generality, we can also assume that $\E |\phi( X)|^q<\infty$ (otherwise the statement of the lemma is trivial).

Let $\Lambda:=\{\omega: X(\omega)\po Y(\omega)\po \wt X(\omega)\}$. Then, by assumption, $\P(\Lambda)\ge1-\eps$ and we get (recall that $d\le 1$)
\begin{align}\label{step1symdist}
\E d(X,Y)&=\E [d(X,Y)\I(\Lambda)]+\E [d(X,Y)\I(\Omega\setminus\Lambda)]\nn\\
&\le\E [(\phi(Y)-\phi(X))\I(\Lambda)]+\P(\Omega\setminus\Lambda)\nn\\
&\le\E [(\phi(\wt X)-\phi(X))\I(\Lambda)]+\eps\nn\\
&=\E (\phi(\wt X)-\phi(X))-\E [(\phi(\wt X)-\phi(X))\I(\Omega\setminus\Lambda)]+\eps.
\end{align}
Since $\Law(X)=\Law(\wt X)$, we see that $\E (\phi(\wt X)-\phi(X))=0$. Applying the Cauchy--Schwarz inequality, we finally obtain from  \eqref{step1symdist}
\begin{align*}
\E d(X,Y)&\le2 \eps^{1/p}(\E |\phi(X)|^q)^{1/q}+\eps.\qedhere
\end{align*}
\end{proof}

\section{Exponential ergodicity of the stochastic reaction--diffusion equation in the hypoelliptic setting}\label{S:3}

Fix a filtered probability space $(\Omega, \F,(\F_t)_{t\ge0},\P)$. We consider the stochastic reaction--diffusion equation \eqref{srde} evolving on a $d$--dimensional torus $\T^d$, $d\in\N$. Note that we have equipped this equation with periodic boundary conditions just to simplify the exposition and to emphasize the main ideas. Everything should work in a similar way for equation \eqref{srde} on a bounded domain equipped with Dirichlet boundary conditions.

We are interested in analytically and probabilistically strong solutions to this equation. Recall the following notions.

\begin{Definition}[{see, e.g., \cite[Definition 1.7]{BG}}]
Let $u_0\in L_2(\T^d)$. An $L_2(\T^d)$-valued, continuous in time process $u\colon\Omega\times[0,\infty)\to L_2(\T^d)$ is called an \textit{analytically generalized strong solution} to \eqref{srde} with  initial condition $u_0$ if $u(0)=u_0$ and for any $\eps>0$, $t>\eps$ the following holds:
\begin{align}
&\int_\eps^t  \|\Delta u(s)\|_{L_2(\T^d)}+\|f(u(s,\cdot),\cdot)\|_{L_2(\T^d)}ds<\infty,\quad \text{$\P$-a.s.};\label{cond1gss}\\
& u(t)=u(\eps)+\int_\eps^t  (\Delta u(s)+f(u(s,\cdot),\cdot))ds+\sum_{k=1}^m \sigma_k (W^k(t)-W^k(\eps)),\quad \text{$\P$-a.s.}.\label{cond2gss}
\end{align}
\end{Definition}

\begin{Definition}[{see, e.g., \cite[Section 6.1]{DaPratoZab}, \cite[Definition 1.3]{BG}}]
Let $u_0\in L_2(\T^d)$. An $L_2(\T^d)$-valued, continuous in time process $u\colon\Omega\times[0,\infty)\to L_2(\T^d)$ is called an \textit{analytically strong solution} to \eqref{srde} with  initial condition $u_0$ if conditions \eqref{cond1gss} and \eqref{cond2gss} hold for $\eps=0$ and any $t>0$.
\end{Definition}

\begin{Definition}[{see, e.g., \cite[Remark 5.2]{Hai09}}]
A solution to \eqref{srde} (strong or generalized strong) is called a \textit{probabilistically strong solution} if it is adapted to the filtration $(\F_t)_{t\ge0}$.
\end{Definition}

In case of ambiguity, a solution to  \eqref{srde} with initial condition $x\in
L_2(\T^d)$ will be denoted by $u^x$. We refer to \cite{BG, DaPratoZab, Hai09} for a detailed discussion of relations between different notions of a solution to SPDEs.

We will make the following assumption on $f$ and $\sigma_i$.
\begin{Assumption}{A}\label{A:1}
The function $f$ is jointly continuous and locally Lipschitz in the first variable. Moreover, the following condition holds (dissipativity outside of a compact set): there exist $K_1, K_2, K_3>0$ such that for any $x,y\in\R$, $\xi\in\T^d$
\begin{align}\label{fone}
&x f(x,\xi)\le K_1-K_2 x^2;\\
\label{fdiff}
&(x-y) (f(x,\xi)-f(y,\xi))\le K_3(x-y)^2.
\end{align}
Suppose that
\begin{equation}\label{sigmaass}
\sigma_k\in\C^2(\T^d),\quad \text{for each $k=1,\hdots,m$}.
\end{equation}
\end{Assumption}

As shown in Theorem~\ref{T:soltorus} below, the condition on $\sigma$ guarantees local existence of strong solution and condition \eqref{fone} prevents blow-up. Conditions \eqref{fone} and \eqref{fdiff} are satisfied, for example, for $f(x,\xi):=K x-x^3$, $K\in\R$.

Recall the definition of the order $\pl$ in \eqref{aaa}.

\begin{Theorem}\label{T:soltorus} Suppose that Assumption~\textbf{\ref{A:1}} holds. Then
\begin{enumerate}[label={\rm(\roman*)}]
 \item for any $u_0\in \C^2(\T^d)$ equation \eqref{srde} has a unique analytically and probabilistically strong solution $u$ with initial condition $u_0$; furthermore $u$ has a version which is continuous on $[0,\infty)\times\T^d$;
 \item for any $u_0\in L_2(\T^d)$ equation \eqref{srde} has a unique analytically generalized strong and probabilistically strong solution $u$ with initial condition $u_0$; furthermore $u$ has a version which is continuous on $(0,\infty)\times\T^d$;
 \item this solution of equation \eqref{srde}  is a homogeneous Markov process with state space $L_2(\T^d)$;
 \item there exists a set $\Omega'\subset\Omega$ of full measure such that if $x,y\in L_2(\T^d)$ and $x\pl y$, then
 $$
 u^x(t,\omega)\pl u^y(t,\omega),\quad t\ge0,\,\omega\in\Omega';
 $$
 \item the corresponding Markov transition function is order preserving with respect to the order $\pl$;
 \item for any $x\in L_2(\T^d)$, $t\ge0$ we have $\E\|u^x(t)\|_{L_2(\T^d)}^2<\infty$; moreover there exists a constant $C>0$ such that the following energy estimates hold for any $x\in L_2(\T^d)$, $0\le s\le t$
\begin{align}
&\E\|u^x(t)\|_{L_2(\T^d)}^2\le \E\|u^x(s)\|_{L_2(\T^d)}^2-K_2\int_s^t\E \|u^x(r)\|_{L_2(\T^d)}^2\, dr\nn\\
&\phantom{\E\|u^x(t)\|_{L_2(\T^d)}^2\le}+(K_1+\|\sigma\|_{L_2(\T^d)}^2)(t-s),
\label{energest1}\\
&\E\|u^x(t)\|_{L_2(\T^d)}^4\le \|x\|_{L_2}^4\exp(-K_2t) +C\label{energest2},
\end{align}
where the constants $K_1$ and $K_2$ were defined in \eqref{fone}.
\end{enumerate}
\end{Theorem}

The proof of the theorem is given in Section~\ref{S:42}. Let us make here a couple of remarks about the theorem. Usually in the literature it is assumed additionally that $f$ is a polynomial or $f$ is globally Lipschitz or that at least some growth bounds on $|f(x,\xi)|$ hold, see, e.g., \cite[Section~7]{DaPratoZab}, \cite[Section~4.2]{BG},  \cite[Section~8.3]{HM11}. Here we established existence and uniqueness of the solutions to \eqref{srde} without these additional restrictions. The main challenge is that the corresponding Nemytskii operator
$$
F(u)(\xi):=f(u(\xi),\xi)
$$
is not $L_2(\T^d)\to L_2(\T^d)$ (or at least $L_p(\T^d)\to L_2(\T^d)$ for some large $p$). Therefore it is difficult to apply here any fixed point principle. Furthermore if $u^n$ is a sequence of solutions to \eqref{srde} with smooth initial conditions $u_0^n$, and $u^n$ converges to some $u$ in $L_2$, then it is not clear at all that this $u$ is a solution to \eqref{srde}. To overcome these obstacles, inspired by some ideas from \cite{BG}, we have  extended the corresponding PDE result \cite[Propositions 7.3.1]{Luna} to $L_2$ initial conditions.

The fact that for irregular initial data ($u_0\in L_2(\T^d)$) equation \eqref{srde} might not have an analytically strong solution is not surprising. Indeed, even for the standard heat equation $\partial_t u=\Delta u$ on $\T^1$ we have $\|\Delta u(t)\|_{L_2(\T^d)}\le
Ct^{-1}\|u(0)\|_{L_2(\T^d)}$ and thus $\int_0^t \|\Delta u(s)\|_{L_2(\T^d)}\,ds$ might be infinite.

Now let us present the main result of this section and establish ergodicity of the stochastic reaction--diffusion equation in the hypoelliptic setting. By Theorem~\ref{T:soltorus}(iii) the solution to \eqref{srde} is a Markov process with the state space $L_2(\T^d)$. Let $(P_t)_{t\ge0}$ be the family of Markov transition probabilities associated with this process. We introduce the following condition.

\begin{Assumption}{B}\label{A:2} There exist $\eps>0$ and $\lambda_k\in\R$, $k=1,\hdots,m$ such that
\begin{equation*}
\sum_{i=1}^m\lambda_k\sigma_k(\xi)>\eps\quad \text{for all $\xi\in\T^d$}.
\end{equation*}
\end{Assumption}

\begin{Theorem}\label{T:invmeasure} Suppose that Assumptions~\textbf{\ref{A:1}} and \textbf{\ref{A:2}} hold. Then the stochastic reaction--diffusion equation \eqref{srde} has a unique invariant measure $\pi$. Furthermore, there exist $C>0$, $\lambda> 0$ such that for all $x\in L_2(\T^d)$ we have
\begin{equation}\label{convratesrde}
W_{\|\cdot\|_{L_2}\wedge1}(P_t(x,\cdot),\pi)\le C(1+\|x\|_{L_2}^2)\exp(-\lambda t),\quad t\ge0.
\end{equation}
\end{Theorem}

\begin{Remark}\label{R:invmeasure}
Assumption~\textbf{\ref{A:2}} is satisfied for example if $m\ge1$ and $\sigma_1(\xi)>\eps$ for all $\xi\in\T^d$.
\end{Remark}

\begin{Remark}
Assumption~\textbf{\ref{A:2}} is imposed only to guarantee the swap condition, see Lemma~\ref{L:Michael}. If the solution to \eqref{srde} satisfies \eqref{swapcond}, then only  Assumptions~\textbf{\ref{A:1}} is needed in Theorem~\ref{T:invmeasure}.
\end{Remark}

As mentioned in the introduction, it is clear that if $m=0$, then equation \eqref{srde} might have multiple invariant measures. On the other hand, if $m=\infty$, equation \eqref{srde} has a unique invariant measure (of course, under certain assumption on the rate of decay of $\|\sigma_k\|_{L_2}$ as $k\to\infty$) \cite[Sections~7 and 11]{DPZ96}. The first result showing uniqueness of the invariant measure for finite $m$ (that is when noise acts not in all directions) is due to Hairer \cite{H02}, who showed it for $m$ large enough. This was later improved by Hairer and Mattingly \cite[Remark 8.22]{HM11}, where uniqueness of the invariant measure and exponential ergodicity is proven if $d=1$ and $m=3$. Remark~\ref{R:invmeasure} refines this result and shows that if the noise acts only in one direction and the corresponding $\sigma_1$ is positive, then the solution to equation \eqref{srde} is exponentially ergodic. Note that contrary to \cite[Remark 8.22]{HM11} we also do not rely here on the specific polynomial form of the drift $f$ and do not assume that the space dimension $d\le 3$ as in \cite[Assumption RD.2]{HM11}.

The next result shows that, in general, the convergence of transition probabilities to the invariant measure in \eqref{convratesrde} might not take place in the  total variation metric, and thus the $W_{\|\cdot\|_{L_2}\wedge1}$ metric can not be replaced in \eqref{convratesrde} by the $d_{TV}$ metric. As discussed in the introduction, this happens due to the fact that the stochastic system do not get ``enough'' noise. A related result for the case when the drift $f$ is linear can be found in \cite[Proposition 9.7]{Hai08}.

\begin{Theorem}\label{T:weareoptimal}
Suppose that Assumption~\textbf{\ref{A:1}} holds for a function $f$ which does not depend on space (so $f=f(x)$, $x\in\R$). Then there exists a function $\sigma$ satisfying additionally   Assumption~\textbf{\ref{A:2}} and an initial condition $x\in L_2(\T^d)$, such that
\begin{equation*}
d_{TV}(P_t(x,\cdot),\pi)=1,
\end{equation*}
for any $t\ge0$.
\end{Theorem}
The proof of Theorem~\ref{T:weareoptimal} is given in Section~\ref{S:42}.

\begin{proof}[Proof of Theorem~\ref{T:invmeasure}]
Let us verify that all  conditions of Theorems~\ref{T:main} and \ref{T:maininv} are satisfied. Let $E=L_2(\T^d)$ with partial order $\pl$ (recall its definition in \eqref{aaa}).  Set
\begin{align}\label{phidef}
&V(x)=\|x\|_{L_2}^2,\,\,d(x,y):=\|x-y\|_{L_2}^2\quad x,y\in L_2(\T^d),\nn\\
&\phi(x):=2\int_{\T^d} x(z)^2\sign(x(z))\,dz,\quad x\in L_2(\T^d).
\end{align}

The Markov semigroup corresponding to \eqref{srde} is order--preserving with respect to $\pl$ by Theorem~\ref{T:soltorus}(v) and thus condition (1) of Theorem~\ref{T:main} holds.  Condition (2) of Theorem~\ref{T:main} is satisfied thanks to energy bound \eqref{energest1}. It follows from Example~\ref{E:main} and the definition of $\phi$ that condition (3) of Theorem~\ref{T:main} is met.
Condition (4) follows from \eqref{energest1} with $M(x):=2\|x\|_{L_2(\T^d)}^4+C$, $x\in L_2(\T^d)$. By Lemma~\ref{L:Michael} condition (5) holds. Conditions (6) and (7) hold trivially with $\delta=1/2$ and  $\kappa=1/2$, respectively. Thus, all conditions of Theorems~\ref{T:main} and \ref{T:maininv} are satisfied. Therefore the process $u$ has a unique invariant measure and \eqref{convratesrde}  holds.
\end{proof}

Finally, we establish exponentially fast synchronization by noise for solutions to \eqref{srde}: that is, we will prove that any two solutions with initial conditions $x,y\in L_2(\T^d)$ launched with the same noise will converge to each other in probability exponentially fast.

\begin{Theorem}\label{T:sync} Suppose that Assumptions~\textbf{\ref{A:1}} and \textbf{\ref{A:2}} hold. Then there exist $C>0$, $\lambda> 0$ such that for any $x,y\in L_2(\T^d)$ we have
\begin{equation}\label{synchratesrde}
\E[\|u^x(t)-u^y(t)\|_{L_2(\T^d)}\wedge1]\le C(1+\|x\|_{L_2(\T^d)}^2+\|y\|_{L_2(\T^d)}^2)\exp(-\lambda t),\quad t\ge0.
\end{equation}
\end{Theorem}
Note that the fact that $\E[\|u^x(t)-u^y(t)\|_{L_2(\T^d)}\wedge1]\to0$ as $t\to\infty$ follows immediately from
Theorem~\ref{T:invmeasure} and \cite[Proposition~2.4]{FGS}. Unfortunately, as discussed in Section~\ref{S:OPanddistance}, the techniques developed in \cite{FGS} do not provide a quantitative bound on convergence rate. Therefore to prove this theorem we use the toolkit developed in Section~\ref{S:OPanddistance}.

\begin{proof}[Proof of Theorem~\ref{T:sync}]
Fix arbitrary $x_0,y_0\in L_2(\T^d)$, $t\ge0$. Define
\begin{equation*}
z_{-}(\xi):=x_0(\xi)\wedge y_0(\xi),\,\,z_{+}(\xi):=x_0(\xi)\vee y_0(\xi),\quad\xi\in\T^d.
\end{equation*}
Put $X:=u^{z_{-}}(t)$, $Y:=u^{z_{+}}(t)$. Then by the order--preserving property (Theorem~\ref{T:soltorus}(iv)) we have
$$
X\pl u^{x_0}(t)\pl Y,\,\,\,\,X\pl u^{y_0}(t)\pl Y,\quad {\text{a.s.}}
$$
Therefore
\begin{equation}\label{step0synch}
[\|u^{x_0}(t)-u^{y_0}(t)\|_{L_2(\T^d)}\wedge1]\le  [\|X-Y\|_{L_2(\T^d)}\wedge1].
\end{equation}

Now let us check that all the conditions of Lemma~\ref{L:33} hold. Let
$E=L_2(\T^d)$ with partial order $\pl$. Set $d(x,y):=\|x-y\|_{L_2}^2\wedge1$, $x,y\in L_2(\T^d)$, and define
$\phi$ as in \eqref{phidef}. Then it follows from Example~\ref{E:main} that condition (3) of Theorem~\ref{T:main} is satisfied. Furthermore, if $d(x,y)=1$, then we have
\begin{equation}\label{phipart1}
|\phi(x)-\phi(y)|\le (|\phi(x)|+|\phi(y)|)\le 2d(x,y)^{1/2}(\|x\|_{L_2}^2+\|y\|_{L_2}^2).
\end{equation}
If $d(x,y)<1$, then one has
\begin{align}\label{phipart2}
|\phi(x)-\phi(y)|\le&  2\int_{\T^d} \bigl|x(z)^2\I(x(z)>0)-y(z)^2\I(y(z)>0)\bigr|\,dz\nn\\
&+2\int_{\T^d} \bigl|x(z)^2\I(x(z)<0)-y(z)^2\I(y(z)<0)\bigr|\,dz\nn\\
\le& 4\int_{\T^d} |x(z)-y(z)|(|x(z)|+|y(z)|)\,dz\nn\\
\le&4\sqrt2 \|x-y\|_{L_2}(\|x\|_{L_2}+\|y\|_{L_2})\nn\\
\le&4\sqrt2 d(x,y)^{1/2}(2+\|x\|^2_{L_2}+\|y\|^2_{L_2}),
\end{align}
where in the second inequality we used the following bound
$$
|a^2\I(a>0)-b^2\I(b>0)\bigr|\le |a-b|(|a|+|b|),\quad a,b\in\R.
$$
Combining \eqref{phipart1} and \eqref{phipart2}, we get
\begin{equation*}
|\phi(x)-\phi(y)|\le d(x,y)^{1/2}(\psi(x)+\psi(y)),
\end{equation*}
where we put $\psi(x):=4\sqrt2(1+\|x\|^2_{L_2})$, $x\in L_2(\T^d)$. Thus, bound \eqref{philipbound} holds.

It is also clear that by definition we have $X\pl Y$. Further, by Theorem~\ref{T:soltorus}(vi) we have
$$
\E |\phi(X)|\le 2\E \|X\|_{L_2}^2=2\E \|u^{z_{-}}(t)\|_{L_2}^2<\infty
$$
and, similarly, $\E |\phi(Y)|<\infty$. Finally, by Theorem~\ref{T:invmeasure} there exist $C,\lambda>0$ such that
\begin{align*}
W_d(\Law(X),\Law(Y))&\le W_{\|\cdot\|_{L_2}\wedge1}(\Law(X),\Law(Y))\\
&\le  C(1+\|z_{-}\|_{L_2}^2+\|z_{+}\|_{L_2}^2)\exp(-\lambda t)\\
&=  C(1+\|x_0\|_{L_2}^2+\|y_0\|_{L_2}^2)\exp(-\lambda t).
\end{align*}
Thus, all the conditions of Lemma~\ref{L:33} are met. Therefore inequality \eqref{mbound33} and energy bound
\eqref{energest2} yield
\begin{align*}
\E d(X,Y)&\le C(1+\|x_0\|_{L_2}^2+\|y_0\|_{L_2}^2)\exp(-\lambda t/2)(1+\sqrt{\E \|X\|_{L_2}^4}+\sqrt{\E \|Y\|_{L_2}^4})\\
&\le C(1+\|x_0\|_{L_2}^2+\|y_0\|_{L_2}^2)\exp(-\lambda t/2)(1+\|z_-\|_{L_2}^2+\|z_+\|_{L_2}^2)\\
&= C(1+\|x_0\|_{L_2}^2+\|y_0\|_{L_2}^2)^2\exp(-\lambda t/2).
\end{align*}
This bound allows us to conclude.
$$
\E \sqrt{d(X,Y)}\le \sqrt{\E d(X,Y)}\le C(1+\|x_0\|_{L_2}^2+\|y_0\|_{L_2}^2)\exp(-\lambda t/4).
$$
Taking into account \eqref{step0synch}, we obtain \eqref{synchratesrde}.
\end{proof}

\section{Proofs}\label{S:Proofs}

\subsection{Proofs of the results of Section~\ref{S:2}}

To prove Theorem~\ref{T:main} we need a couple of lemmas. The first lemma is quite standard and is in the spirit of  \cite[Section~15.2]{MT}. The lemma deals with  the connection between the finiteness of the exponential moment of the first return time of a Markov process to a set and the existence of a Lyapunov function. However we were not able to find in the literature precisely this statement (we found only a number of related ones). Thus, for the convenience of the reader and for the sake of completeness we provide the full proof of the lemma.

Let us consider a measurable space $(\wt E, \wt{\mathcal{E}})$ and a Markov kernel $Q$ on it. Let $Z=(Z_n)_{n\in\Z_+}$ be a Markov process with transition kernel $Q$.  For a set $A\in\wt{\mathcal{E}}$ introduce the first return time to $A$
\begin{equation*}
\tau_A:=\inf\{n\ge1\colon Z_n\in A\}.
\end{equation*}

\begin{Lemma}\label{L:4.2}
Suppose there exists a measurable function $\V\colon \wt E\to [1,\infty)$ such that for some $\lambda\in(0,1)$, $K>0$
\begin{equation}\label{lyapcond}
Q\V\le \lambda \V + K.
\end{equation}
Fix
\begin{equation}\label{eqM}
M>(K/(1-\lambda))\vee1
\end{equation}
and put $r:=(\lambda+K/M)^{-1}$.
\begin{enumerate}[label={\rm(\roman*)}]
\item
We have
\begin{align}
&\E_x r^{\tau_{\{\V\le M\}}}\le\V(x),\quad  x\in\wt E\setminus\{\V\le M\};\label{expboundrettime1}\\
&\E_x r^{\tau_{\{\V\le M\}}}\le rQ\V(x),\quad  x\in \{\V\le M\}.\label{expboundrettime2}
\end{align}
\item
Further, suppose that for a  set $A\in\wt{\mathcal{E}}$ we have for some $\eps>0$
\begin{equation}\label{infcond}
\inf_{x\in\{\V\le M\}}Q(x,A)\ge\eps.
\end{equation}
Then, for any $1<l<r^{\frac{|\log(1-\eps)|\wedge\log M}{2\log M}}$ there exists a constant $C=C(\lambda,l,K,M,\eps)>0$  such that for any $x\in\wt E$
\begin{equation}\label{expboundA}
 \E_x l^{\tau_A}\le C \V(x).
\end{equation}

\end{enumerate}

\end{Lemma}
\begin{proof}
(i). The proof uses the standard argument. First note, that it follows from the Lyapunov condition \eqref{lyapcond} and the definition of $r$ that
\begin{equation}\label{newlyapcond}
Q\V(x)\le r^{-1}\V(x),\quad x\in\wt E\setminus\{\V\le M\}.
\end{equation}

Put now for $n\in\Z_+$
$$
\tau^{(n)}:=\tau_{\{\V\le M\}}\wedge n \wedge \inf\{k\in\Z_+\colon r^k \V(Z_k)\ge n\}.
$$
Then by Dynkin's formula for discrete time Markov chains (see, e.g., \cite[Theorem~11.3.1]{MT}) we have for any $x\in\wt E$
\begin{equation}\label{Dfirst}
\E_x  r^{\tau^{(n)}}\V(Z_{\tau^{(n)}})=\V(x)+\E_x\Bigl[\sum_{i=1}^{n}r^{i-1}\I(i-1<\tau^{(n)})(\E[r\V(Z_i)|Z_{i-1}]-\V(Z_{i-1}))\Bigr].
\end{equation}
If $x\in\wt E\setminus\{\V\le M\}$, then by definition for any $i\in[1,\tau^{(n)}]$ we have  $Z_{i-1}\in\wt E\setminus\{\V\le M\}$. Therefore, \eqref{newlyapcond} implies for any $i\in[1,n]$
\begin{equation*}
r^{i-1}\I(i-1<\tau^{(n)})(\E[r\V(Z_i)|Z_{i-1}]-\V(Z_{i-1}))\le  0.
\end{equation*}
Combining this with \eqref{Dfirst}, we get
\begin{equation*}
\E_x  r^{\tau^{(n)}}\le \E_x  r^{\tau^{(n)}}\V(Z_{\tau^{(n)}})\le \V(x),\quad x\in\wt E\setminus\{\V\le M\},
\end{equation*}
which by Fatou's lemma yields \eqref{expboundrettime1}.

If $x\in \{\V\le M\}$, then by the Markov property and  bound \eqref{expboundrettime1},
\begin{equation*}
\E_x  r^{\tau_{\{\V\le M\}}}\!= r \P_x (Z_1\in\{\V\le M\}) + r\E_x [\V(Z_1)\I (Z_1\notin\{\V\le M\})]\le r \E_x \V(Z_1)=r Q\V(x),
\end{equation*}
which is \eqref{expboundrettime2}.

(ii).  Fix $l\in(1,r^{\frac{|\log(1-\eps)|\wedge\log M}{2\log M}})$.
For $n\in\Z_+$ denote by $\tau^M(n)$ the $n$th return time to the set $\{\V\le M\}$:
\begin{align*}
&\tau^M(1):=\inf\{n\ge0\colon Z_n\in\{\V\le M\}\};\\
&\tau^M(n):=\inf\{n\ge \tau^M(n-1)+1\colon Z_n\in\{\V\le M\}\},\quad n\ge2.
\end{align*}
Introduce random variables
$$
I_n:=\I(Z_{\tau^M(n)+1}\in A),\quad n\in\Z_+.
$$
We see that the event $\{I_n=1\}$ corresponds to the reaching of the desired set $A$ after $n$th visit to the set $\{\V\le M\}$. We have for any $x\in\wt E$
\begin{align}\label{part12step1}
\E_x l^{\tau_A}&\le\sum_{k=1}^\infty \E_x [l^{\tau^M(k)+1}\I(I_1=...=I_{k-1}=0, I_k=1)]\nn\\
&\le\sum_{k=1}^\infty \E_x [l^{\tau^M(k)+1}\I(I_1=...=I_{k-1}=0)].
\end{align}
Define
$$
D(M,l):=\sup_{x\in\{\V\le M\}}\E_x [l^{\tau_{\{\V\le M\}}}\I(Z_1\notin A)].
$$
Note that by the strong Markov property and the definition above, for any $x\in\wt E$
\begin{align}\label{part12step2}
&\E_x [l^{\tau^M(k)+1}\I(I_1=...=I_{k-1}=0)]\nn\\
&\qquad=\E_x \bigl[l^{\tau^M(k-1)+1}\I(I_1=...=I_{k-2}=0)\E_{Z_{\tau^M(k-1)}} (l^{\tau_{\{\V\le M\}}}\I(Z_1\notin A))\bigr]\nn\\
&\qquad\le D(M,l)\E_x [l^{\tau^M(k-1)+1}\I(I_1=...=I_{k-2}=0)]\nn\\
&\qquad\le D(M,l)^{k-1}\E_x [l^{\tau^M(1)+1}]\nn\\
&\qquad\le l  D(M,l)^{k-1}  (\V(x)\vee M),
\end{align}
where in the last inequality we used \eqref{expboundrettime1} and the fact that $l<r$. Applying the Cauchy--Schwarz inequality and \eqref{expboundrettime2}, we deduce
\begin{equation}\label{DML}
D(M,l)\le \sup_{x\in\{\V\le M\}}(\E_x [l^{2\tau_{\{\V\le M\}}}])^{1/2}(1-\eps)^{1/2}\le M^{\log l/\log r}(1-\eps)^{1/2}<1,
\end{equation}
where the last inequality follows from the definitions of $l$ and $r$. Thus, \eqref{DML} together with  \eqref{part12step1}, \eqref{part12step2}, and the obvious bound $(\V(x)\vee M)\le M\V(x)$ yields
\begin{align*}
\E_x l^{\tau_A}&\le  l  (\V(x)\vee M) \sum_{k=1}^\infty D(M,l)^{k-1}\le C(\lambda,l,K,M,\eps) \V(x),
\end{align*}
which implies \eqref{expboundA}.
\end{proof}

The next lemma explains why conditions \ref{Cond:tretij} and \ref{Cond:chetv} are imposed in Theorem~\ref{T:main}.

\begin{Lemma}\label{L:41}
Assume that condition \ref{Cond:tretij} of Theorem~\ref{T:main} is satisfied for measurable functions $\phi\colon E\to\R$, $d\colon E\times E\to[0,1]$.
Let $X,\, Y,\, \wt X,\, \wt Y$ be random elements $\Omega\to E$ such that $\Law(X)=\Law(\wt X)$, $\Law(Y)=\Law(\wt Y)$ and
$$
\P(X\po Y)\ge1-\eps_1,\,\,\P(\wt Y\po \wt X)\ge1-\eps_2
$$
for some $\eps_1,\eps_2>0$. Then
$$
\E d(X,Y)\le 2\sqrt{\eps_1+\eps_2}(\E \phi( X)^2)^{1/2}+\eps_1+\eps_2.
$$
\end{Lemma}
\begin{proof}
It follows from the gluing lemma (see, e.g., \cite[p.~23]{Villani}) that there exist random elements $Z_1, Z_2, Z_3$
such that $\Law(Z_1,Z_2)=\Law(X,Y)$ and $\Law(Z_2,Z_3)=\Law(\wt Y,\wt X)$. Clearly, we also have $\Law(Z_1)=\Law(Z_3)$.

Let $\Lambda$ be the set $\{\omega: Z_1(\omega)\po Z_2(\omega)\po Z_3(\omega)\}$. Then by above
\begin{equation*}
\P(\Lambda)\ge \P(Z_1\po Z_2)+\P(Z_2\po Z_3)-1= \P(X\po Y)+\P(\wt Y\po \wt X)-1\ge1-\eps_1-\eps_2.
\end{equation*}
Thus, we immediately get from Lemma~\ref{L:44} that
\begin{equation*}
\E d(X,Y)=\E d(Z_1,Z_2)\le 2\sqrt{\eps_1+\eps_2}(\E \phi( X)^2)^{1/2}+\eps_1+\eps_2.\qedhere
\end{equation*}
\end{proof}

\begin{proof}[Proof of Theorem~\ref{T:main}]
The proof of the theorem consists of several steps.
\newcounter{step}
\refstepcounter{step}

\textbf{Step \arabic{step}.} In this step we fix $x,y\in E$ and $t>0$. Let $\{X^x(s),s\ge0\}$ and $\{X^y(s), s\ge0\}$  be independent Markov processes  with the laws $\P_x$ and $\P_y$, correspondingly.
Let $\F_s:=\sigma(X^x(r),X^y(r),r\le s)$, $s\le t$, be the natural filtration.
Introduce stopping times
\begin{align*}
&\tau_{x\po y}:=\inf\{n\in\Z_+: X^x(n)\po X^y(n)\},\\
&\tau_{y\po x}:=\inf\{n\in\Z_+: X^y(n)\po X^x(n)\}.
\end{align*}
Note that by definition these stopping times $\tau_{x\po y}$ and $\tau_{y\po x}$ take only countably many values.

Let us now extend the state space $E$ and add an additional element denoted by $\ae$. We assume that $\ae\po\ae$ and do not impose any further partial order relations between $\ae$ and other elements of $E$.

Introduce the following random elements:
\begin{align*}
&\eta^x:=X^x(t)\I(\tau_{x\po y}\le t)+\ae \I(\tau_{x\po y}>t),\\
&\eta^y:=X^y(t)\I(\tau_{x\po y}\le t)+\ae \I(\tau_{x\po y}>t).
\end{align*}
We claim now that $\Law(\eta^x)\pm\Law(\eta^y)$. Indeed, let $f\colon E\cup\{\ae\}\to\R$ be an arbitrary bounded measurable increasing function. Then
\begin{equation}\label{order1}
\E f(\eta^x)=f(\ae)\P(\tau_{x\po y}>t)+\sum_{i=0}^{\lfloor t\rfloor}\E [f(X^x(t))
\I(\tau_{x\po y}=i)]
\end{equation}
Note that for any $i=0,..,\lfloor t\rfloor$ we have
\begin{align}\label{order2}
\E [f(X^x(t))\I(\tau_{x\po y}=i)]&=\E \bigl[\I(\tau_{x\po y}=i)\E \bigl(f(X^x(t))|\F_i\bigr)\bigr]\nn\\
&=\E \Bigl[\I(\tau_{x\po y}=i)\int_E f(z)P_{t-i}(X^x(i),dz)\Bigr].
\end{align}
Recall that the kernel $P_{t-i}$ is order preserving and thus for any $z_1,z_2\in E$
such that $z_1\po z_2$ we have
\begin{equation*}
\int_E f(z)P_{t-i}(z_1,dz)\le \int_E f(z)P_{t-i}(z_2,dz).
\end{equation*}
Since on the set $\{\tau_{x\po y}=i\}$ we have $X^x(i)\po X^y(i)$, we can continue \eqref{order2} in the following way:
\begin{align*}
\E [f(X^x(t))\I(\tau_{x\po y}=i)]&\le\E \Bigl[\I(\tau_{x\po y}=i)\int_E f(z)P_{t-i}(X^y(i),dz)\Bigr]\\
&=\E [f(X^y(t))\I(\tau_{x\po y}=i)].
\end{align*}
Combining this with \eqref{order1}, we finally deduce
\begin{equation*}
\E f(\eta^x)\le f(\ae)\P(\tau_{x\po y}>t)+\sum_{i=0}^{\lfloor t\rfloor}\E [f(X^y(t))
\I(\tau_{x\po y}=i)]=\E f(\eta^y).
\end{equation*}
Since $f$ was an arbitrary bounded measurable increasing function, we see that indeed
$\Law(\eta^x)\pm\Law(\eta^y)$.

\textbf{Step \refstepcounter{step}\label{step2}\arabic{step}}.  We have shown that $\Law(\eta^x)\pm\Law(\eta^y)$. Therefore, by the Strassen theorem (see, e.g., \cite[Theorem IV.2.4]{L92}) there exist random variables $\xi^x$ and $\xi^y$ such that $\Law(\xi^x)=\Law(\eta^x)$, $\Law(\xi^y)=\Law(\eta^y)$ and $\xi^x\po \xi^y$ a.s.

It follows from the construction, that
$\P(\eta^x\neq X^x(t))\le \P(\tau_{x\po y}>t)$ and $\P(\eta^y\neq X^y(t))\le \P(\tau_{x\po y}>t)$. Now we will apply twice the gluing lemma (see, e.g., \cite[p.~23]{Villani}). First, we apply it to the pairs $(X^x(t),\eta^x)$ and
$(\xi^x,\xi^y)$ and construct two random variables $Y^x$ and $Y^y$ such that
$\Law (Y^x)=\Law (X^x(t))$, $\Law (Y^y)=\Law (\xi^y)=\Law (\eta^y)$ and $\P(Y^x\po Y^y)\ge
 1-\P(\tau_{x\po y}>t)$. Then, we apply the gluing lemma to the pairs $(Y^x,Y^y)$ and $(\eta^y,X^y)$. We deduce that there exist random variables $Z^x$ and $Z^y$ such that
\begin{align*}
&\Law (Z^x)=\Law (X^x(t)),\quad \Law (Z^y)=\Law (X^y(t));\\
&\P(Z^x\po Z^y)\ge 1-2\P(\tau_{x\po y}>t).
\end{align*}
In a similar way, we construct another pair of random variables $\wt Z^x$ and $\wt Z^y$ with the following properties
\begin{align*}
 &\Law (\wt Z^x)=\Law (X^x(t)),\quad \Law (\wt Z^y)=\Law (X^y(t));\\
 &\P(\wt Z^y\po \wt Z^x)\ge 1-2\P(\tau_{y\po x}>t).
\end{align*}

\textbf{Step \refstepcounter{step}\arabic{step}.}
Now it follows from Step~\ref{step2} and Lemma~\ref{L:41} that
\begin{align}\label{part2impstep}
W_{d\wedge1}(P_t(x,\cdot),P_t(y,\cdot))&\le \E (d(Z^x, Z^y)\wedge1)\nn\\
&\le 2\sqrt2 (\P(t\le\tau_{y\po x})+\P\bigl(t\le\tau_{x\po y})\bigr)^{1/2}(1+M(x)^{1/2}).
\end{align}
Thus everything boils down to bounding the probabilities $\P(t\le\tau_{y\po x})$ and $\P\bigl(t\le\tau_{x\po y})$.
We will do it using Lemma~\ref{L:4.2}.

\textbf{Step \refstepcounter{step}\arabic{step}.}
First we note that, clearly,  $\{(x_1,x_2)\in E\times E\colon x_1\po x_2\}\supset A\times B$ (recall the definitions of the sets $A$ and $B$ in condition~\ref{Cond:pyat} of the theorem).  Therefore,
\begin{equation}\label{taubound}
\tau_{x\po y}\le \inf\{n\in \N: (X^x(n),X^y(n))\in A\times B\}.
\end{equation}
Now we apply Lemma~\ref{L:4.2}(ii) to the state space $\wt E:=E\times E$, the kernel $Q$ on it defined by
$$
Q((x_1,x_2),A_1\times A_2):=P_1(x_1,A_1)P_1(x_2,A_2),\quad x_1,x_2\in E,\,A_1,A_2\in \mathcal{E},
$$
the set $A\times B$, and the Lyapunov function $\V(x_1,x_2):=1+V(x_1)+V(x_2)$, $x_1,x_2\in E$. It follows from \eqref{contlyap} and the Gronwall inequality that
\begin{equation*}
Q\V(x_1,x_2)\le e^{-\gamma}\V(x_1,x_2)+(2K/\gamma+1)(1-e^{-\gamma}).
\end{equation*}
Therefore, condition \eqref{lyapcond} is met. Take $M=4K/\gamma+1$. It is clear  that this choice of $M$ satisfies \eqref{eqM}. Thanks to assumption \eqref{nonmixing} and the definitions of $\V$ and the sets $A$, $B$, we see
that condition \eqref{infcond} holds for the set $A\times B$ in place of $A$ and the positive constant $M$ chosen above. Therefore all conditions of Lemma~\ref{L:4.2}(ii) are satisfied and, thus there exist constants $C>0$, $\lambda>0$ that do not depend on $x,y$ such that
\begin{equation*}
\E_{x,y} e^{\lambda \tau_{A\times B}}\le C(1+V(x)+V(y)).
\end{equation*}
This, combined with \eqref{taubound} and the Chebyshev inequality implies
\begin{equation*}
\P(t\le \tau_{y\po x})\le C(1+V(x)+V(y)) e^{-\lambda t}.
\end{equation*}
Using exactly the same argument, we also get that
\begin{equation*}
\P(t\le \tau_{x\po y})\le C(1+V(x)+V(y)) e^{-\lambda t}.
\end{equation*}
Substituting these bounds into \eqref{part2impstep}, we finally deduce
\begin{equation*}
W_{d\wedge1}(P_t(x,\cdot),P_t(y,\cdot))\le C (1+V(x)+V(y))(1+M(x))e^{-\lambda t/2}.
\end{equation*}
Taking into account that obviously $W_{d\wedge1}(P_t(x,\cdot),P_t(y,\cdot))\le1$, we obtain \eqref{finalres}.
\end{proof}

\begin{proof}[Proof of Theorem~\ref{T:maininv}]
We begin by observing that by Jensen's inequality and condition~\ref{Cond:dist} of the theorem we have for any measures $\mu,\nu\in\mathcal{P}(E)$
\begin{equation}\label{Jensen}
W_{\rho\wedge1}(\mu,\nu)\le W_{d^\delta\wedge1}(\mu,\nu)\le (W_{d\wedge1}(\mu,\nu))^\delta.
\end{equation}

We apply now Theorem~\ref{T:main} with $\theta:=\frac{\kappa}{\delta(1+\kappa)}$. Taking into account \eqref{Jensen}, we obtain that there exist $C>0$, $\lambda>0$ such that for any $x,y\in E$ we have
\begin{equation}\label{boundnomer1}
W_{\rho\wedge1}(P_t(x,\cdot),P_t(y,\cdot))\le C(1+V(x)+V(y))e^{-\lambda t},\quad t\ge0.
\end{equation}
where we have also used condition~\ref{Cond:MV} of the theorem.

From here the proof follows the standard line of argument. Note that by \cite[Theorem~17.24]{K12} for any fixed $t\ge0$,  the mapping $x\mapsto P_t(x,\cdot)$ is measurable.  Therefore by \cite[Theorem 4.8 and Corollary 5.22]{Villani} the mapping $(x,y)\mapsto W_{\rho\wedge1}(P_t(x,\cdot),P_t(y,\cdot))$ is measurable and the function $W_{\rho\wedge1}$ is convex in both arguments. Therefore we immediately  get from \eqref{boundnomer1} that for any $\mu,\nu\in\mathcal{P}(E)$
\begin{equation}\label{Pochtigood}
W_{\rho\wedge1}(P_t \mu,P_t \nu)\le C(1+\mu(V)+\nu(V))e^{-\lambda t},\quad t\ge0.
\end{equation}

Now we are ready to prove the existence of an invariant measure for $\{P_t,t\ge0\}$.  Fix any $x\in E$. We make use of \eqref{Pochtigood} to obtain that for any $t,s>0$
\begin{align*}
W_{\rho\wedge1}(P_t(x,\cdot),P_{t+s}(x,\cdot))&=W_{\rho\wedge1}(P_t \delta_x,P_{t}(P_s\delta_x))\\
&\le C(1+V(x)+P_sV(x))e^{-\lambda t}\\
&\le C(1+2V(x)+K/\gamma)e^{-\lambda t}.
\end{align*}
Thus the sequence of measures $\{P_t(x,\cdot)\}$ is Cauchy in the space $(\mathcal{P}(E),W_{\rho\wedge1})$. Since this space is complete (see, e.g., \cite[Theorem~6.18]{Villani}), there exists a measure $\pi\in\mathcal{P}(E)$ such that
$W_{\rho\wedge1}(P_t(x,\cdot),\pi)\to0$ as $t\to\infty$. Hence for any $y\in E$ we have
\begin{equation*}
W_{\rho\wedge1}(P_t(y,\cdot),\pi)\le W_{\rho\wedge1}(P_t(x,\cdot),\pi)+W_{\rho\wedge1}(P_t(x,\cdot),P_t(y,\cdot))\to0\,\,\text{as $t\to\infty$}.
\end{equation*}
Therefore for any $t\ge0$ by the monotone convergence theorem we deduce
\begin{equation*}
W_{\rho\wedge1}(P_t \pi,\pi)\le \int_E W_{\rho\wedge1}(P_t(y,\cdot),\pi)\,\pi(dy)\to0\,\,\text{as $t\to\infty$},
\end{equation*}
where we have used again that  the function $W_{\rho\wedge1}$ is convex. Therefore the measure $\pi$ is invariant for $\{P_t,t\ge0\}$.

Now let us show the uniqueness of the invariant measure. First, note that if $\pi$ is an arbitrary invariant measure for $\{P_t,t\ge0\}$, then $\pi(V)<\infty$, see \cite[Proposition 4.24]{HaiCon}. Thus if $\pi_1$ and $\pi_2$ are two invariant measures, then by \eqref{Pochtigood} we have for any $t\ge0$
\begin{equation*}
W_{\rho\wedge1}(\pi_1, \pi_2)=W_{\rho\wedge1}(P_t \pi_1,P_t \pi_2)\le C(1+\pi_1(V)+\pi_2(V))e^{-\lambda t}.
\end{equation*}
Taking the limit in the right--hand side of the above inequality we get $\pi_1=\pi_2$.

Finally, bound \eqref{finalres2} follows directly from \eqref{Pochtigood} and the fact that $\pi(V)<\infty$.
\end{proof}

\begin{proof}[Proof of Example~\ref{E:main}]
First let us show that the set $\Gamma$ defined in \eqref{closedset} is closed. Let $(x_n)_{n\in\Z_+}$,
$(y_n)_{n\in\Z_+}$ be two sequences of elements of $L_p(D,\R)$ such that $x_n\pl y_n$ and $\|x_n-x\|_{L_p(D)}\to0$, $\|y_n-y\|_{L_p(D)}$ as $n\to\infty$. We claim now that $x\pl y$.

Indeed, by passing to an appropriate subsequence $(n_k)$ we see that $x_{n_k}\to x$ almost everywhere (a.e.) and $y_{n_k}\to y$ a.e. as $n\to\infty$. Since for each $k$ we have $x_{n_k}\le y_{n_k}$ a.e., it is now easily seen that $x\le y$ a.e. and thus $x\pl y$. Therefore, the set $\Gamma$ is closed.

Now let us introduce the function $\phi$. Put
\begin{equation*}
\phi(x):=2^{p-1}\int_D \bigl|x(z)|^p\sign(x(z))\,dz,\quad x\in L_p(D,\R).
\end{equation*}

Let us verify that the partial order $\pl$ and functions $\phi$ and  $d$ satisfy condition \ref{Cond:tretij}
of Theorem~\ref{T:main}. Our original proof of this was a bit long; the following simple proof was suggested to us by the referee, to whom we are very grateful for this.

Fix any $x,y\in L_p(D,\R)$ such that $x\pl y$. Note that for any real numbers $a<b$, $p\ge1$ one has
$$
|a-b|^p\le 2^{p-1}(|b|^p\sign(b) - |a|^p\sign(a)).
$$
Using this inequality, we get
\begin{align*}
0\le d(x,y)&=\|x-y\|_{L_p(D)}^p\\
&\le 2^{p-1}  \int_D \bigl(|y(z)|^p \sign(y(z))-|x(z)|^p\sign(x(z))\,dz\\
&=\phi(y)-\phi(x),
\end{align*}
and thus condition \ref{Cond:tretij}
of Theorem~\ref{T:main} holds.
\end{proof}

\subsection{Proofs of the results of Section~\ref{S:3}}\label{S:42}

We begin with the following auxiliary statement which provides Gronwall-type bounds that will be very useful in the sequel.
\begin{Lemma}\label{L:SGR} Let $C_1,C_2,C_3,C_4$ be positive constants. Let $X$ be a continuous $\F_t$-adapted process taking values in $[0,\infty)$ such that $X(0)=x$ and
\begin{equation}\label{grontype}
X(t)\le X(s)-C_1\int_s^tX(r)dr+M(t)-M(s)+C_2(t-s),\quad 0\le s\le t,
\end{equation}
where $M$ is a continuous local $\F_t$-martingale with $M(0)=0$ and
\begin{equation}\label{quadrcov}
d\langle M\rangle_t\le (C_3 X(t)+C_4)dt.
\end{equation}
Then for any $t\ge0$ we have $\E X(t)<+\infty$ and the following bounds hold:
\begin{align}
&\E X(t)\le \E X(s)-C_1\int_s^t\E X(r)dr+C_2(t-s),\quad 0\le s\le t,\label{grontype1}\\
&\E X(t)^2\le x^2 \exp(-C_1t)+C,\quad t\ge0,\label{grontype2}
\end{align}
where $C=C(C_1,C_2,C_3,C_4)>0$ is some constant independent of $t$ and $x$.
\end{Lemma}
\begin{proof}
We begin with the proof of \eqref{grontype1}. For $n\in\N$ introduce the stopping time $\tau_n:=\inf\{t\ge0: |M(t)|\ge n\}$. Then it follows from \eqref{grontype} that for any $n\in\N$, $t\ge0$
\begin{equation*}
\E X(t\wedge\tau_n )\le x+C_2 (t\wedge\tau_n).
\end{equation*}
Using Fatou's lemma, we deduce that for any $t\ge0$ we have $\E X(t)\le x+C_2 t$ and thus $\E X(t)<\infty$. This and \eqref{quadrcov} imply that for any $t\ge0$ we have $\E\langle M\rangle_t\le C_4t+C_3xt+C_2C_3t^2/2<\infty$. Therefore $M$ is a martingale and \eqref{grontype1} follows immediately from \eqref{grontype}.

To establish \eqref{grontype2}, we introduce a process $Y$, which is a solution to the following equation
\begin{equation}\label{equationY}
Y(t)= x-C_1\int_0^tY(r)dr+M(t)+C_2t,\quad t\ge0.
\end{equation}
We claim that for any $t\ge0$ we have $X(t)\le Y(t)$. Indeed, assume the contrary and suppose that for some $T_0>0$ we have $X(T_0)> Y(T_0)$. Then, arguing as in \cite[Proof of Proposition 9.2]{HeM15}, we introduce
$$
S_0:=\sup\{t\in[0,T_0]\colon X(t)\le Y(t)\}.
$$
Since the processes $X$ and $Y$ are continuous we have $S_0<T_0$ and $X(t)> Y(t)$ for $t\in(S_0,T_0)$. Then
\eqref{grontype} and \eqref{equationY} imply
\begin{equation*}
X(T_0)-Y(T_0)\le X(S_0)-Y(S_0)-C_1\int_{S_0}^{T_0}(X(t)-Y(t))\,dt<0
\end{equation*}
which contradicts the fact that $X(T_0)> Y(T_0)$. Therefore such $T_0$ does not exist and $X(t)\le Y(t)$
for any $t\ge0$.

Note now that by Ito's formula and \eqref{quadrcov}
\begin{align}\label{YIto}
d Y(t)^2&=(-2C_1Y(t)^2+2C_2Y(t))dt+d\langle M\rangle_t+dN(t)\nn\\
&\le (-2C_1Y(t)^2+(2C_2+C_3)Y(t)+C_4)dt+dN(t)\nn\\
&\le (-C_1Y(t)^2+C_5)dt+dN(t),
\end{align}
where $C_5:=(2C_2+C_3)^2/(4C_1)+C_4$ and $N$ is a continuous local martingale.

Consider a stopping time $\sigma_n:=\inf\{t\ge0: |N(t)|\ge n\}$. Then, using \eqref{YIto}, we get
for any $0\le s\le t$
\begin{equation*}
\E Y(t\wedge\sigma_n)^2\le\E Y(s\wedge\sigma_n)^2-C_1\int_{s}^t \E Y(r\wedge\sigma_n)^2\,dr+C_5(t-s).
\end{equation*}
Therefore, the Gronwall inequality yields
\begin{equation*}
\E X(t\wedge\sigma_n)^2\le\E Y(t\wedge\sigma_n)^2\le x^2 \exp(-C_1t)+C_5/C_1,\quad t\ge0.
\end{equation*}
The proof of \eqref{grontype2} is completed by passing to the limit in the above inequality using Fatou's lemma.
\end{proof}

In all the remaining theorems and lemmas of this section we will assume that Assumption~\textbf{\ref{A:1}} is satisfied.
To establish Theorem~\ref{T:soltorus} we split the process $u$ into the following two parts. The first part, denoted by $w$, is the stochastic convolution, that is a unique analytically and probabilistically strong solution of the following equation
\begin{equation}\label{stochconv}
d w(t,\xi)=\Delta w(t,\xi)dt+\sum_{k=1}^m \sigma_k(\xi)dW^k(t), \quad \xi\in \T^d,\, t\ge0,
\end{equation}
with the initial condition $w(0)=0$.

\begin{Lemma}\label{L:stochconv} The function $w$ is well--defined. Furthermore, there exists a set $\Omega'\subset \Omega$ such that $\P(\Omega')=1$ and  for any $\eps>0$, $T>0$, $\omega\in\Omega'$ there exists $C=C(\eps,T,\omega)$ such that
\begin{equation}\label{regularityw}
|w(t_1,\xi_1,\omega)-w(t_2,\xi_2,\omega)|\le C(|t_1-t_2|^{1/2-\eps}+|\xi_1-\xi_2|^{1-\eps}),\quad \xi_1,\xi_2\in\T^d,\, t_1,t_2\in[0,T].
\end{equation}
\end{Lemma}
\begin{proof}
Existence and uniqueness of a strong solution of \eqref{stochconv} follows from \cite[Example~5.39]{DaPratoZab} and assumption \eqref{sigmaass}. Bound \eqref{regularityw} follows from \cite[Theorem~5.15(ii)]{DaPratoZab}.
\end{proof}

For $T>0$ denote now
\begin{equation}\label{MDef}
M_T=M_T(\omega):=\sup_{\substack{t\in[0,T]\\\xi\in\T^d}}|w(t,\xi,\omega)|
\end{equation}
By Lemma~\ref{L:stochconv}, for any $\omega\in\Omega'$, $T>0$ we have $M_T(\omega)<\infty$.

The second ingredient of decomposition of $u$ is much  more tricky. We fix now arbitrary $\omega\in\Omega'$, $T>0$ and consider the (deterministic) PDE
\begin{equation}\label{equationv}
d v(t,\xi)=\Delta v(t,\xi)dt+f(v(t,\xi)+w(t,\xi,\omega),\xi)dt, \quad \xi\in \T^d,\, t\in[0,T].
\end{equation}
with the initial condition $v(0)=v_0\in \C^2(\T^d)$. We begin with the case when the initial condition is smooth enough.

\begin{Lemma}\label{L:pde} For any $v_0\in \C^2(\T^d)$  equation \eqref{equationv} has a unique analytically strong solution $v$ with the initial condition $v_0$. Furthermore $v\in\C([0,T],\C^2(\T^d))$ and there exists $C>0$ such that for any $t\in[0,T]$
 \begin{equation}\label{boundlinfty}
 \|v(t)\|_{L_\infty(\T^d)}\le  C(\|v_0\|_{L_2(\T^d)}+M_t)(1+t^{-d/4})+C_ft,
 \end{equation}
where $C_f:=\sup_{\substack{x\ge0\\\eta\in\T^d}}(\sign(x)f(x,\eta))\vee0$.
\end{Lemma}

\begin{proof}
We begin with the uniqueness part. Suppose that $v$ and $\bar v$ are two strong solutions to equation \eqref{equationv} with the same initial condition $v_0\in \C^2(\T^d)$. Then, by the chain rule
\begin{align}\label{uniq}
&\frac{d}{dt}\|v(t)-\bar v(t)\|_{L_2(\T^d)}^2\nn\\
&\quad=2\int_{\T^d}(v(t,\xi)-\bar v(t,\xi))\bigl(f(v(t,\xi)+w(t,\xi,\omega),\xi)-f(\bar v(t,\xi)+w(t,\xi,\omega),\xi)\bigr)\,d\xi\nn\\
&\quad\phantom{=} -2\|\grad (v(t)-\bar v(t))\|_{L_2(\T^d)}^2\nn\\
&\quad\le C \|v(t)-\bar v(t)\|_{L_2(\T^d)}^2,
\end{align}
where the last inequality follows from assumption \eqref{fdiff}. Since $v(0)=\bar v(0)$, an application of the Gronwall lemma immediately implies that $v(t)=\bar v(t)$ for any $t\in[0,T]$.

To show existence of a strong solution to \eqref{equationv} we fix the initial condition $v_0\in \C^2(\T^d)$ and  $T>0$. Note that the function
\begin{equation}\label{functlipsh}
\R\times[0,T]\times\T^d\ni(x,t,\xi)\mapsto f(x+w(t,\xi,\omega),\xi)
\end{equation}
is locally Lipschitz in $x$, H\"older in $t$ and continuous in $\xi$ thanks to \eqref{regularityw} and Assumption~\textbf{\ref{A:1}}. Therefore it satisfies \cite[Inequality (7.3.2)]{Luna}. Thus, by \cite[Propositions 7.3.1.i, 7.3.1.ii, 7.1.10.iii]{Luna} equation \eqref{equationv} has a local solution on some interval $[0,\delta]$, $\delta>0$ and this solution is in $\C([0,\delta],\C^2(\T^d))$. Let us now show that
equation \eqref{equationv} has a global solution.

It follows from \eqref{fone} that there exist $x_-\le  x_+\in\R$ such that $f(x,\xi)\ge0$ for any $x\le x_-$, $\xi\in\T^d$ and
$f(x,\xi)\le0$ for any $x\ge x_+$, $\xi\in\T^d$. Consider the following PDE
\begin{equation}\label{PDEpsi}
d \psi(t,\xi)=\Delta \psi(t,\xi)dt+(f(\psi(t,\xi)+w(t,\xi,\omega),\xi)\vee0)dt, \quad \xi\in \T^d,\, t\ge0,
\end{equation}
with the initial condition $\psi(0)=(x_++M_T)\vee\sup_\eta v_0(\eta)$ (recall the definition of $M_T$ in \eqref{MDef}). By above, the constant function
$$
\psi(t,\xi):=(x_++M_T)\vee\sup_\eta v_0(\eta)
$$
solves this equation on  $[0,T]$. On the other hand, since $v(0,\xi)\le\psi(0,\xi)$ and $f(x+w(t,\xi,\omega),\xi)\le
f(x+w(t,\xi,\omega),\xi)\vee0$, the comparison theorem for reaction--diffusion PDEs \cite[Theorem~10.1]{Smoller} implies that for every interval $[0,\delta]$ where equation \eqref{equationv} has a local solution we have
$$
v(t,\xi)\le (x_++M_T)\vee\sup_\eta v_0(\eta),\quad t\in[0,\delta],\, \xi\in\T^d.
$$
By a similar argument,
$$
v(t,\xi)\ge (x_--M_T)\wedge\inf_\eta v_0(\eta),\quad t\in[0,\delta],\, \xi\in\T^d.
$$
Thus, for any interval $[0,\delta]$ where equation \eqref{equationv} has a local solution, this solution is uniformly bounded by a constant which does not depend on $\delta$. Therefore, by \cite[Proposition 7.3.1.v]{Luna}
equation \eqref{equationv} has a global solution on  $[0,T]$. By \cite[Proposition 7.1.10.iii]{Luna} this solution is in $\C([0,T],\C^2(\T^d))$.

Finally, to obtain the desired bound on $\|v(t)\|_{L_\infty(\T^d)}$ in terms of $\|v_0\|_{L_2(\T^d)}$ rather than
$\|v_0\|_{L_\infty(\T^d)}$ we fix $t\in[0,T]$ and consider again PDE \eqref{PDEpsi} with the initial condition $\psi(0,\xi):=(v_0(\xi)\vee0)+M_t$. By the comparison principle, we have
\begin{equation}\label{uest}
v(t,\xi)\le \psi(t,\xi), \quad \xi\in\T^d.
\end{equation}
Since $\psi(0,\xi)\ge M_t$ for all $\xi\in\T^d$ and the drift is nonnegative, it is immediate to see that
$\psi(s,\xi)\ge M_t$ for all $s\in[0,t]$, $\xi\in\T^d$. Therefore, taking into account \eqref{MDef}, we get
that for all $s\in[0,t]$, $\xi\in\T^d$
\begin{equation}\label{compareCplus}
f(\psi(s,\xi)+w(s,\xi,\omega),\xi)\vee0\le 0\vee\sup_{\substack{x\ge0\\\eta\in\T^d}}f(x,\eta)=:C_+<\infty.
\end{equation}
Introduce now $\psi_+$, which is the strong solution of the following PDE
\begin{equation*}
d \psi_+(s,\xi)=\Delta \psi_+(s,\xi)ds+C_+ds, \quad \xi\in \T^d,\, s\in[0,t],
\end{equation*}
with the initial condition $\psi_+(0):=\psi(0)$. Using \eqref{compareCplus}, we see that the comparison principle implies that
\begin{equation}\label{psiest}
\psi(t,\xi)\le \psi_+(t,\xi), \quad \xi\in\T^d.
\end{equation}
Let $p_s$ be the heat kernel on the torus $\T^d$. Then it is straightforward to see that for any $\xi\in\T^d$
\begin{equation*}
|\psi_+(t,\xi)|=|p_t*\psi_0(\xi)+C_+t|\le C_+t+\|\psi_0\|_{L_2}\Bigl(\int_{\T^d} p_t(\xi)^2\,d\xi\Bigr)^{1/2}\le C_+t+C\|\psi_0\|_{L_2}(1+t^{-d/4}).
\end{equation*}
Combining this with \eqref{uest} and \eqref{psiest}, we deduce
$$
v(t,\xi)\le C_+t+C\|\psi_0\|_{L_2}(1+t^{-d/4})\le t(\sup_{\substack{x\ge0\\\eta\in\T^d}}f(x,\eta)\vee0)+C(\|v_0\|_{L_2}+M_t)(1+t^{-d/4})
$$
By a similar argument, we get
$$
v(t,\xi)\ge t(\inf_{\substack{x\le0\\\eta\in\T^d}}f(x,\eta)\wedge0)-C(\|v_0\|_{L_2}+M_t)(1+t^{-d/4}).
$$
This yields \eqref{boundlinfty}.
\end{proof}

Now let us move on to  less regular initial data.

\begin{Lemma}\label{L:Linfty} For any $v_0\in L_\infty(\T^d)$  equation \eqref{equationv} has a unique analytically generalized strong solution $v$ with initial condition $v_0$. Furthermore, we have $v\in\C((0,T],\C^2(\T^d))$.
\end{Lemma}
\begin{proof}
We begin with the uniqueness part. Suppose that $v$ and $\bar v$ are two analytically generalized strong solutions to equation \eqref{srde} with the same initial condition $v_0\in L_\infty(\T^d)$. Fix arbitrary $t>0$. Arguing as in the proof of Lemma~\ref{L:pde}, we see that by the Gronwall lemma there exists $C=C(t)>0$ such that for any $\delta\in(0,t)$ we have
\begin{equation}\label{uniqgen}
\|v(t)-\bar v(t)\|_{L_2(\T^d)}^2\le C \|v(\delta)-\bar v(\delta)\|_{L_2(\T^d)}^2.
\end{equation}
By definition, $v(\delta)\to v(0)$ and $\bar v(\delta)\to \bar v(0)$ in $L_2(\T^d)$ as $\delta\to0$. Since $v(0)=\bar v(0)$, by passing to the limit as $\delta\to0$ in \eqref{uniqgen}, we deduce that $v(t)=\bar v(t)$.

Note again that the function defined in \eqref{functlipsh} is locally Lipschitz in $x$, H\"older in $t$ and continuous in $\xi$. Therefore, by \cite[Proposition~7.3.1.i]{Luna} there exists $\delta>0$ such that on the interval $[0,\delta]$ equation \eqref{equationv} has an analytically generalized strong solution $v$ with  initial condition $v_0$ and
$v\in\C((0,\delta],\C^2(\T^d))$. Since $v(\delta)\in\C^2(\T^d)$, by Lemma~\ref{L:pde} we can construct
an analytically strong solution $v$ on $[\delta,T]$ with the initial condition $v(\delta)$ and $v\in\C([\delta,T],\C^2(\T^d))$. By gluing these two solutions together we get an analytically generalized strong solution $v\in\C((0,T],\C^2(\T^d))$.
\end{proof}

To consider even less regular initial data (recall that we are interested in the initial conditions from $L_2(\T^d)$) we need the following lemma about approximations of solutions to \eqref{equationv}.

\begin{Lemma}\label{L:approx}
Let $(v^n_0)_{n\in\Z_+}$ be a sequence of $\C^2(\T^d)$ functions converging in $L_2$ to $v_0\in L_2(\T^d)$. Let $v^n$ be the analytically strong solution of \eqref{equationv} with the initial condition $v^n_0$. Then
\begin{enumerate}[label={\rm(\roman*)}]
\item $v^n$ converges in $\C([0,T],L_2(\T^d))$ to some $v\in\C([0,T],L_2(\T^d))$;
\item for each $t\in(0,T]$ we have $v(t)\in L_\infty(\T^d)$;
\item if $\bar v$ is an analytically generalized strong solution of \eqref{equationv} with the initial condition $v_0$, then for each $t\in[0,T]$ we have $\bar v(t,\xi)=v(t,\xi)$ for almost all $\xi\in\T^d$.
.
\end{enumerate}
\end{Lemma}

\begin{proof} (i). Let $n,m\in\Z_+$. Then arguing as in the derivation of \eqref{uniq} we get by the chain rule and assumption \eqref{fdiff}
\begin{equation*}
\frac{d}{dt}\|v^n(t)-v^m(t)\|^2_{L_2(\T^d)}\le C\|v^n(t)-v^m(t)\|^2_{L_2(\T^d)}.
\end{equation*}
By the Gronwall lemma, this implies that
\begin{equation*}
\sup_{t\in[0,T]}\|v^n(t)-v^m(t)\|^2_{L_2(\T^d)}\le C(T)\|v^n_0-v^m_0\|^2_{L_2(\T^d)}.
\end{equation*}
Now the statement of the lemma follows  immediately from the completeness of the space $\C([0,T],L_2(\T^d))$  (\cite[Theorem I.4.19]{K12}).

(ii). Fix $t\in(0,T]$. By Lemma~\ref{L:pde}, we have the following uniform over $n\ge N$ bound, where $N$ is large enough:
\begin{equation}\label{univ}
\|v^n(t)\|_{L_\infty(\T^d)}\le  C(t)(\|v^n_0\|_{L_2(\T^d)}+M_t)+C_ft\le
2C(t)(\|v_0\|_{L_2(\T^d)}+M_t)+C_ft.
\end{equation}
Since $v^n(t)$ converges to $v(t)$ in $L_2$, bound \eqref{univ} implies $v(t)\in L_\infty(\T^d)$.

(iii). Let $\bar v$ be an analytically generalized strong solution of \eqref{equationv} with the initial condition $v_0$. Then arguing as in part (i) of the lemma, we get for any $\delta>0$
\begin{equation*}
\sup_{t\in[\delta,T]}\|\bar v(t)-v^n(t)\|^2_{L_2(\T^d)}\le C(T)\|\bar v(\delta)-v^n(\delta)\|^2_{L_2(\T^d)}.
\end{equation*}
By passing to the limit as $\delta\to0$ and using the fact that by definition $\bar v(\delta)\to v_0$ in $L_2$, we get
\begin{equation*}
\sup_{t\in[0,T]}\|\bar v(t)-v^n(t)\|^2_{L_2(\T^d)}\le C(T)\|v_0-v^n_0\|^2_{L_2(\T^d)}.
\end{equation*}
However, by part (i) of the lemma
\begin{equation*}
\sup_{t\in[0,T]}\| v(t)-v^n(t)\|^2_{L_2(\T^d)}\le C(T)\|v_0-v^n_0\|^2_{L_2(\T^d)}.
\end{equation*}
This implies that $\bar v(t)=v(t)$ as elements of $L_2(\T^d)$ for any $t\in[0,T]$.
\end{proof}

The next lemma establishes existence of solutions to equation \eqref{equationv} with $L_2(\T^d)$ initial data. It relies on Lemma~\ref{L:approx} and extends
\cite[Propositions 7.3.1]{Luna}.

\begin{Lemma}\label{L:L2} For any $v_0\in L_2(\T^d)$  equation \eqref{equationv} has a unique analytically generalized strong solution $v$ with the initial condition $v_0$. Furthermore, we have $v\in\C((0,T],\C^2(\T^d))$.
\end{Lemma}
\begin{proof}
The proof of the uniqueness part is the same as in Lemma~\ref{L:Linfty}.

Let us show existence of a solution to \eqref{equationv}. Let $(v^n_0)_{n\in\Z_+}$ be a sequence of $\C^2(\T^d)$ functions converging in $L_2$ to $v_0\in L_2(\T^d)$. Let $v^n$ be the analytically strong solution of \eqref{equationv} with the initial condition $v^n_0$ (it exists thanks to Lemma~\ref{L:pde}). By Lemma~\ref{L:approx}(i,ii) there exists $v\in\C([0,T],L_2(\T^d))$ such that 
\begin{equation}\label{allproperties}
v(t)\in L_{\infty}(\T^d),\,\,t\in(0,T] \quad\text{and }\sup_{t\in[0,T]}\|v^n(t)-v(t)\|_{L_2}\to0\,\,\text{as $n\to\infty$}.
\end{equation}
We claim now that $v$ is an analytically generalized strong solution to \eqref{equationv} with the initial condition $v_0$.

Indeed, we have by construction $v(0)=v_0$ and $v\in\C([0,T],L_2(\T^d))$ as a limit of the functions from the space $\C([0,T],L_2(\T^d))$. Fix any $\eps\in(0,T]$ and let us verify that identity \eqref{cond2gss} holds.

By \eqref{allproperties}, we have $v(\eps/2)\in L_\infty$. Therefore, by Lemma~\ref{L:L2} there exists a process $\bar v\in\C((\eps/2,T],\C^2(\T^d))$, which is an analytically generalized strong solution of \eqref{equationv} on interval $[\eps/2,T]$ with the initial condition $v(\eps/2)$. Therefore, by Lemma~\ref{L:approx}(iii) we have
$\bar v(t)=v(t)$ for any $t>\eps/2$. Thus, identity \eqref{cond2gss} holds and for any $t\ge\eps$ the function $v(t)$ has a version which is in $\C^2(\T^d)$. Since $\eps$ was arbitrary, this implies the statement of the lemma.
\end{proof}

Now we are ready to present the proof of Theorem~\ref{T:soltorus}.

\begin{proof}[Proof of Theorem~\ref{T:soltorus}]
(i). The proof of the uniqueness part is the same as in Lemma~\ref{L:pde}. To show existence of a strong solution to \eqref{srde} we fix an initial condition $u_0\in \C^2(\T^d)$, $T>0$ and first show existence on the time interval $[0,T]$. Let $v$ be an analytically strong solution to \eqref{equationv} with the initial condition $v(0)=u_0$. Recall the definition of $\Omega'$ from Lemma~\ref{L:stochconv} and put now
$$
u(t,\xi,\omega):=v(t,\xi,\omega)+w(t,\xi,\omega),\quad t\in[0,T], \xi\in\T^d,  \omega\in\Omega'
$$
and $u(t,\xi,\omega)=0$ for $\omega\in\Omega\setminus\Omega'$. It follows from Lemmas~\ref{L:stochconv} and \ref{L:pde} that the function $u$ is an analytically strong solution to \eqref{srde} with the initial condition $u_0$. To show the  adaptiveness of $u$, we introduce a function $u^n$, $n\in\Z_+$, which is an analytically and probabilistically strong solution on $[0,T]$ to the following equation
\begin{equation}\label{un}
u^n(t,\xi)=u_0(\xi)+\int_0^t [\Delta u^n(s,\xi)+f^n(u^n(s,\xi),\xi)]dt+\sum_{k=1}^m \sigma_k(\xi)W^k(t), \quad \xi\in \T^d,\, t\ge0,
\end{equation}
where for $x\in\R$, $\xi\in\T^d$ we put $f^n(x,\xi):=f((x\wedge n)\vee(-n),\xi)$. Since $f^n$ is uniformly bounded, it follows from \cite[Chapter II, Theorem 2.1]{KR} that $u^n$ is well--defined; thus, identity \eqref{un} holds on some set $\Omega_n$ of full measure. On the other hand, by uniqueness, we have
$$
A_n:=\cap_{k=1}^\infty\Omega_k\cap\Omega'\cap\{\sup_{t\in[0,T],\xi\in\T^d}|u(t,\xi)|\le n\}=
\cap_{k=1}^\infty\Omega_k\cap\Omega'\cap\{\sup_{t\in[0,T],\xi\in\T^d}|u^n(t,\xi)|\le n\}
$$
and $u^n=u$ on $A_n$. Since for each $\omega\in\Omega'$ the function $u$ is bounded, we see that $\P(A_n)\to1$ as $n\to\infty$ and $A_n\subset A_{n+1}$.  This implies that $u^n$ converges to $u$ a.s. as $n\to\infty$. Since each
$u^n$ is $(\F_t)$--adapted, their limit $u$ is also  $(\F_t)$--adapted. Therefore $u$ is a probabilistically strong solution of \eqref{srde} on $[0,T]$. Since $T$ was arbitrary, this and uniqueness imply that $u$ is an analytically and probabilistically strong solution of \eqref{srde} on $[0,\infty)$. The continuity of $u$ follows from continuity of $v$ and $w$ (Lemmas~\ref{L:stochconv} and \ref{L:pde}).

(ii). The proof of the uniqueness part is the same as in Lemma~\ref{L:Linfty}. To show existence fix $T>0$ and
put again now
$$
u(t,\xi,\omega):=v(t,\xi,\omega)+w(t,\xi,\omega),\quad t\in[0,T], \xi\in\T^d,  \omega\in\Omega'
$$
and $u(t,\xi,\omega)=0$ for $\omega\in\Omega\setminus\Omega'$. Here $v$ is an analytically generalized strong solution to \eqref{equationv} with the initial condition $v(0)=u_0$. It follows from Lemmas~\ref{L:stochconv} and \ref{L:Linfty} that the function $u$ is an analytically generalized strong solution to \eqref{srde} with the initial condition $u_0$. To show the adaptiveness we consider $(u^n_0)_{n\in\Z_+}$, a sequence of $\C^2(\T^d)$ functions converging in $L_2$ to $u_0\in L_2(\T^d)$. Let $u^n$ be the analytically and probabilistically strong  solution of \eqref{srde} with the initial condition $u^n_0$ (it exists thanks to part (i) of the theorem). By
Lemma~\ref{L:approx}(iii), we have that for each $t>0$ the function $u^n(t)$ converges to $u(t)$ in $L_2(\T^d)$ $\P$-a.s. Since $u^n(t)$ is $\F_t$--adapted, we see that $u(t)$ is also  $(\F_t)$--adapted. Therefore $u$ is a probabilistically strong solution of \eqref{srde} on $[0,T]$. The proof ends in the same way as in the part (i) of the theorem.

(iii). We begin with the following two observations. First, we note that part (ii) of the theorem implies that the generalized strong solution of equation \eqref{srde} is unique.

Second, for $x\in L_2(\T^d)$ let us denote by $u^x$ the generalized strong solution of \eqref{srde} with the initial condition $x$.  Then, by  the chain rule and Gronwall's inequality we get that for each $t>0$ there exists $C(t)>0$ such that for any $x,y \in L_2(\T^d)$ we have
\begin{equation*}
\|u^x(t)-u^y(t)\|^2_{L_2(\T^d)}\le C(t)\|x-y\|^2_{L_2(\T^d)},
\end{equation*}
where we also took into account assumption \eqref{fdiff}.

Now the Markov property of $u$ follows from these two properties by repeating literally the argument from \cite[Proposition 4.1]{BS}.

(iv). Take the set $\Omega'$ defined in Lemma~\ref{L:stochconv}. Fix $t>0$, $\omega\in\Omega'$, $x,y\in L_2(\T^d)$ such that $x\pl y$. Let $u^x$, $u^y$ be the generalized strong solutions of \eqref{srde} with the initial conditions $x$ and $y$ respectively. Let us show now that $u^x(t,\omega)\pl u^y(t,\omega)$ a.s.

We begin with the case when $x,y\in \C^2(\T^d)$. Let $w$ be the stochastic convolution (recall its definition in Lemma~\ref{L:stochconv}). Let $v^x$, $v^y$ be analytically strong solutions of \eqref{equationv} with the initial conditions $x$ and $y$ respectively. Then by Lemma~\ref{L:pde} $v^x$ and $v^y$ are continuous on $[0,t]\times \T^d$. Therefore, by comparison principle \cite[Theorem~10.1]{Smoller} we have $v^x(t,\omega)\pl v^y(t,\omega)$ and hence $u^x(t,\omega)=v^x(t,\omega)+w(t,\omega)\pl v^y(t,\omega)+w(t,\omega)=u^y(t,\omega)$.

In the general case, we consider $(x^n)_{n\in\Z_+}$ and $(y^n)_{n\in\Z_+}$, two sequences of $\C^2(\T^d)$ functions converging in $L_2$ to $x$ and $y$, respectively. By the above, for each $n\in\Z_+$ we have
\begin{equation}\label{posn}
u^{x^n}(t,\omega)\pl u^{y^n}(t,\omega).
\end{equation}
On the other hand, by part (ii) of the theorem and Lemma~\ref{L:approx} we have $u^{x^n}(t,\omega)\to u^{x}(t,\omega)$, $u^{y^n}(t,\omega)\to u^{y}(t,\omega)$ in $L_2(\T^d)$ as $n\to\infty$. This together with \eqref{posn} yields $u^{x}(t)\pl u^{y}(t)$.

(v). Follows immediately from part (iv) of the theorem and the Strassen's theorem.

(vi) Fix $x\in L_2(\T^d)$. Then $u^x$ is an analytically generalized strong solution of \eqref{srde}. Therefore by Ito's lemma, taking into account \eqref{fone}, we deduce for any $0<s\le t$
\begin{align}\label{energyeststep1}
\|u^x(t)\|_{L_2(\T^d)}^2\le& \|u^x(s)\|_{L_2(\T^d)}^2-2K_2\int_s^t \|u^x(r)\|_{L_2(\T^d)}^2\, dr+(2K_1+\|\sigma\|_{L_2(\T^d)}^2)(t-s)\nn\\
&+M(t)-M(s),
\end{align}
where the constants $K_1$, $K_2$ are defined in \eqref{fone} and $M$ is a continuous local martingale with
$M(0)=0$ and
$$
d\langle M\rangle_t=4\sum_{k=1}^m\int_{\T^d}(u^x(t,\xi))^2\sigma_k(\xi)^2\,d\xi dt\le
4\|u^x(t)\|_{L_2(\T^d)}^2 \sum_{k=1}^m\|\sigma_k\|^2_{L_\infty(\T^d)}dt.
$$
Using the fact that $u^x$ is continuous in $L_2(\T^d)$, we can pass to the limit in \eqref{energyeststep1} as $s\to0$ to deduce that \eqref{energyeststep1} is also valid for $s=0$.

Therefore all the conditions of Lemma~\ref{L:SGR} are satisfied for the process $X(t):=\|u^x(t)\|_{L_2(\T^d)}^2$. Thus, $\E\|u^x(t)\|_{L_2(\T^d)}^2<\infty$ for any $t>0$, inequality \eqref{energest1} follows from \eqref{grontype1} and inequality \eqref{energest2} follows from \eqref{grontype2}.
\end{proof}

Before we formulate and prove the final lemma we recall that under Assumption~\textbf{\ref{A:1}}
the solution $u$ with initial condition $u_0 \in L_2(\T^d)$ satisfies the following
mild form of the SPDE \eqref{srde} \cite[Theorem 5.4]{DaPratoZab}
\begin{align*}
  \begin{split}
  u(t,\xi)=&\int_{\T^d} p_t(\xi-\zeta)u(\varepsilon,\zeta)d\zeta + \sum_{k=1}^m\int_\varepsilon^t \int _{\T^d} p_{t-s}(\xi-\zeta)\sigma_k(\zeta)\,d\zeta \,dW^k(s)\\
  &+\int_\varepsilon^t \int _{\T^d} p_{t-s}(\xi-\zeta) f(u(s,\zeta),\zeta)\,d\zeta\,ds,
\end{split}
\end{align*}
for each $t >0$, $\varepsilon \in (0,t]$  and $\xi \in \T^d$, where $p$ denotes the fundamental solution of the heat equation on $\T^d$.
Letting $\varepsilon \downarrow 0$, we see that almost surely
\begin{align}\label{mild0}
  \begin{split}
  u(t,\xi)=&\int_{\T^d} p_t(\xi-\zeta)u_0(\zeta)d\zeta + \sum_{k=1}^m\int_0^t \int _{\T^d} p_{t-s}(\xi-\zeta)\sigma_k(\zeta)\,d\zeta \,dW^k(s)\\
  &+\lim_{\varepsilon \downarrow 0}\int_\varepsilon^t \int _{\T^d} p_{t-s}(\xi-\zeta) f(u(s,\zeta),\zeta)\,d\zeta\,ds,
  \end{split}
\end{align}
where we used the Cauchy-Schwarz inequality to obtain convergence of the first term.

For an arbitrary element $x\in L_2(\T^d)$ let $S_x$ and $L_x$ be the sets of elements
from $L_2(\T^d)$ that are smaller (respectively larger) than $x$, that is
\begin{equation*}
S_{x}:=\{f\in L_2(\T^d): f\po x\},\,\, L_x:=\{f\in L_2(\T^d): x \po f\}.
\end{equation*}
It is easy to see that the sets $S_{x}$ and $L_x$ are closed. For $a \in \R$ put
$$
\gg{a}(\xi):=a,\quad \xi\in\T^d.
$$
\begin{Lemma}\label{L:Michael}
Suppose that Assumptions~\textbf{\ref{A:1}} and \textbf{\ref{A:2}} hold. Let $(P_t)$ be the semigroup associated with equation \eqref{srde}.
Then for any $M>0$ there exists $\eps>0$ such that
\begin{equation}\label{swapcond}
\inf_{\substack{x\in L_2(\T^d)\\\|x\|_{L_2}\le M}} P_1(x,S_{\gg 0})>\eps,\quad\quad
\inf_{\substack{x\in L_2(\T^d)\\\|x\|_{L_2}\le M}} P_1(x,L_{\gg 0})>\eps.
\end{equation}
\end{Lemma}

\begin{proof} By symmetry it suffices to prove the first claim. We show it in two steps.

\medskip
\textbf{Step 1.} At this step we will use only Assumption~\textbf{\ref{A:1}}. We claim that for every $M>0$ and $T\in (0,1]$ there exist $\Gamma \in \R$ and $t_0\in (0,T]$ such that
  $$
\inf_{\substack{x\in L_2(\T^d)\\\|x\|_{L_2}\le M}} P_{t_0}(x,S_{\gg \Gamma})>0.
$$

First we note, that it is sufficient to prove the claim only for large enough $M$. Observe now that by \eqref{fone} there exists  $\gamma \in \R$ such that $f(z,\xi)\le 0$ for all $z \ge \gamma$ and all $\xi \in \T^d$. Fix arbitrary $\beta>\gamma$, and assume that $M$ is large enough so that $\|\gg \beta\|_{L_2}\le M$.
By order preservation (Theorem \ref{T:soltorus} (iv)) it  suffices to show that for every $M>0$, $T \in (0,1]$  there exist
$\Gamma \in \R$ and $t_0\in (0,T]$ such that
$$
\inf_{\substack{x\in L_2(\T^d), \gg \beta\pl x\\\|x\|_{L_2}\le M}} P_{t_0}(x,S_{\gg \Gamma})>0.
$$

By continuity of solutions (Theorem  \ref{T:soltorus}) there exists  $t_0 \in (0,T]$ such that
$$
\kappa:=\P\big(\gg \gamma \pl u^{\gg \beta}(s)  \mbox{ for all } s \in [0,t_0]\big)>0.
$$
For $x \in L_2(\T^d)$ we
define
$$
\Omega(x):=\big\{\omega: \,  \gg \gamma\pl u^x(s) \mbox{ for all } s\in [0,t_0]\big\}.
$$

Fix now arbitrary $x\in L_2(\T^d)$ satisfying $\gg \beta\pl x$ and $\|x\|_{L_2}\le M$. Then, using order preservation again, we see that
$$
\P\big( \Omega(x)\big) \ge \kappa.
$$
Choose $\tilde \Gamma$ so large that
$$
\P\big(\tilde \Omega\big)\!:=\P\Big(\sum_{k=1}^m \int_0^r\!\int_{\T^d}p_{r-s}(\xi-\zeta)\sigma_k(\zeta)\,d\zeta \,dW^k(s) \le \tilde \Gamma \mbox{ for all } 0 \le r \le t_0,\, \xi \in \T^d\Big)\ge 1 - \frac{\kappa}2.
$$
Using \eqref{mild0}, we see that on the set $\tilde \Omega \cap \Omega (x)$ (which has measure at least $\kappa/2$) the solution $u^x$ satisfies for all $\xi\in\T^d$
$$
u^x(t_0,\xi)\le \Big(\int_{\T^d}p_{t_0}^2(\xi-\zeta)\,d\zeta\Big)^{1/2}\|x\|_{L_2}+ \tilde \Gamma \le \Big(\int_{\T^d}p_{t_0}^2(\xi-\zeta)\,d\zeta\Big)^{1/2}M+\tilde \Gamma =:\Gamma,
$$
where we used the Cauchy--Schwarz inequality in the first inequality and the fact that $f(z,\xi)\le 0$ for $z \ge \gamma$ and $\xi \in \T^d$.  The proof of the claim of Step 1 is complete.

\medskip
\textbf{Step 2.} This is the only part of the proof of Theorem~\ref{T:invmeasure}, where Assumption~\textbf{\ref{A:2}} is used. We claim that for every $\Gamma \in \R$ and $\tau>0$ we have
$$
P_\tau(\gg \Gamma,S_{\gg 0})>0.
$$

To show this claim, first of all, we assume, without loss of generality, that $\Gamma>0$.
It suffices to show that for some $\tau_0=\tau_0(\Gamma)>0$ the statement holds for every $\tau \in (0,\tau_0]$. Fix 
numbers $\lambda_k$, $k=1,...,m$ as in
Assumption~\textbf{\ref{A:2}} with $\eps=1$ and define $\Lambda(\xi):=\sum_{k=1}^m\lambda_k\sigma_k(\xi)\ge 1$ and $\hat \Lambda:=\max_{\xi \in \T^d}\Lambda(\xi)$.
Choose $Q > \hat \Lambda (\Gamma +2)+2$ and fix $\tau_0\in (0,1]$ such that $\tau_0|f(z,\xi)|\le 1$ for all $\xi \in \T^d$ and $z \in [\Gamma-Q,\Gamma +2]$ (such $\tau_0$ exists since $f$ is continuous). Fix
$\tau \in (0,\tau_0]$ and $\Theta:=\frac {\Gamma+2}{\tau}$. Then, for
$$
\wt W^k(s):=\Theta \lambda_k s +W^k(s),\; s \ge 0, \,k\in\{1,...,m\}
$$
and $t>0$, the solution $u^{\gg \Gamma}$  satisfies
\begin{align*}
  \begin{split}
  u^{\gg \Gamma}(t,\xi)=&\Gamma - \int_0^t\int_{\T^d}p_{t-s}(\xi-\zeta)\Lambda(\zeta)\,\Theta \,d\zeta\,ds + \sum_{k=1}^m \int_0^t\int_{\T^d}p_{t-s}(\xi-\zeta)\sigma_k(\zeta)\,d\zeta \,d\wt W^k(s)\\
  &+\int_0^t\int_{\T^d}p_{t-s}(\xi-\zeta)f(u^{\gg \Gamma}(s,\zeta),\zeta)\,d\zeta \,ds\\
  =:&\Gamma-I_1(t,\xi)+I_2(t,\xi,\omega)+I_3(t,\xi,\omega).
\end{split}
\end{align*}
Note that for each $t \ge 0$,
\begin{equation*}
I_1(t,\xi)\in [\Theta t,\Theta \hat \Lambda t].
  \end{equation*}
Let
$$
\hat \Omega:=\{\omega:\,|I_2(s,\xi,\omega)|\le 1 \mbox{ for all } s \in [0,1],\xi \in \T^d\}.
$$
Then $\P (\hat \Omega)>0$. Let $\psi:=\inf\{s \ge 0: \,u^{\gg \Gamma}(s)\notin[\widecheck{\Gamma-Q} ,\widecheck {\Gamma +2}]\}$.
If $\psi<\tau$ and    $u(\psi,\xi)=\Gamma+2$ for some $\xi\in\T^d$, then, on the set $\hat \Omega \cap \{\psi<\tau\}$
we have
$$
\Gamma+2\le \Gamma - \Theta \psi+1+1,
$$
so $\psi=0$ which is impossible (since $u$ is continuous).
If, on the other hand, $\psi<\tau$ and    $u(\psi,\xi)=\Gamma-Q$ for some $\xi\in\T^d$, then, on the set $\wt \Omega \cap \{\psi<\tau\}$,
we have
$$
\Gamma-Q \ge \Gamma-\Theta \hat\Lambda \psi  -2,
$$
which is also impossible by the definition of $Q$ and $\Theta$.

Thus, $\psi \ge \tau$ almost surely on $\hat \Omega$. Therefore, on $\hat \Omega$, we have
$$
u^{\gg \Gamma}(\tau,\xi)\le \Gamma - \Theta   \tau+2\le 0
$$
for every $\xi \in \T^d$, so the proof of Step 2 is complete.

\medskip
Using order preservation once more, we see that Step 1 and Step 2 imply the assertion in the lemma.
\end{proof}

Finally, before we present the proof of Theorem~\ref{T:weareoptimal}, we need to introduce some additional notation. If for some $a\in\R$, $x\in L_2(\T^d)$ we have $x(\xi)=a$ for all $\xi\in \T^d$, then we will write
$$
\Pi x = a;\,\,\, \Pi^{-1}a =x.
$$
Introduce also the set of all constants in the space $L_2(\T^d)$, that is,
$$
A_c:=\{\Pi^{-1}a\mid a\in\R\}\subset L_2(\T^d).
$$

Let $(P_t)_{t\ge0}$ be the transition function associated with the solution to \eqref{srde}.
\begin{Lemma}\label{L:pochtifin} Consider equation \eqref{srde} with $m=1$ and $\sigma_1\equiv1$. Suppose that the drift $f$ does not depend on space. This equation has the following properties:
\begin{enumerate}
\item[(i)] if $u(0)\in A_c$, then $u(t)\in A_c$ a.s. for any $t\ge0$;
\item[(ii)] if $u(0)\notin A_c$, then $u(t)\notin A_c$ a.s. for any $t\ge0$;
\item[(iii)] this equation has a unique invariant measure $\pi$ and $\pi(A_c)=1$.
\end{enumerate}
\end{Lemma}
\begin{proof}
(i) Suppose that $u(0)\in A_c$. Let $a=\Pi [u(0)]$. Consider a stochastic differential equation
\begin{equation}\label{SDE}
X_t^a=a+\int_0^t f(X^a_s)\,ds+W_t.
\end{equation}
Since the function $f$ is locally Lipschitz and satisfies the one-sided growth condition, equation \eqref{SDE} has a unique strong solution \cite[Proposition 2.1(b)]{SchSch}. It is immediate to see that the function $u(t,\xi):=X^a(t)$ is an analytically and probabilistically strong solution of \eqref{srde} with the initial condition $u(0)$. By Theorem~\ref{T:soltorus} this is the unique solution to \eqref{srde}. Thus $u(t)\in A_c$ for any $t\ge0$.

(ii) Fix the initial condition $x\notin A_c$. First, consider the case $x\in\C^2(\T^d)$. Let $a_1,a_2$ be real numbers such that $a_1\le x(\xi)\le a_2$ for all $\xi\in\T^d$. Let $\Omega''\subset \Omega$ be the set of full measure where the trajectories of the Brownian motion $W$ are continuous and identity \eqref{cond2gss} holds for $\eps=0$ for the initial conditions $\Pi^{-1}a_1$, $\Pi^{-1}a_2$, and $x$. Set $\Omega_0:=\Omega''\cap\Omega'$, where $\Omega'$ is defined in Theorem~\ref{T:soltorus}(iv). It follows that $\P(\Omega_0)=1$.

Assume now the contrary. That is suppose that  for some $T>0$, $\omega\in\Omega_0$ we have $u^x(0)=x$ and $u^x(T,\omega)\in A_c$. Denote $b:=\Pi [u^x(T,\omega)]$. By Theorem~\ref{T:soltorus}(iv) and part(i) of the theorem, we have for all $\xi\in\T^d$
$$
X_T^{a_1}(\omega)=u^{\Pi^{-1}a_1}(T,\xi,\omega)\le u^{x}(T,\xi,\omega) \le u^{\Pi^{-1}a_2}(T,\xi,\omega)=X_T^{a_2}(\omega).
$$
It follows immediately from the Gronwall lemma, the fact that $f$ is locally Lipschitz, and the comparison principle that solutions to \eqref{SDE} are continuous with respect to the initial condition. Therefore there exists some initial condition $a\in(a_1,a_2)$ such that $X^a(0,\omega)=a$ and $X^a(T,\omega)=b$. Denote now $v(t,\xi):=u(T-t,\xi,\omega)-X^a(T-t,\omega)$, $t\in[0,T]$, $\xi\in\T^d$. Then we have $v(0,\xi)=0$ for any $\xi\in\T^d$. On the other hand, for any $t\in[0,T]$, $\xi\in\T^d$ we have
\begin{align}\label{precarl}
&|\Delta v(t,\xi)+\partial_t v(t,\xi)|\nn\\
&\qquad=|\Delta u(T-t,\xi,\omega)-f(X^a(T-t,\omega))-\Delta u(T-t,\xi,\omega)-f(u(T-t,\xi,\omega))|\nn\\
&\qquad\le C|X^a(T-t,\omega)-(u(T-t,\xi,\omega))|\nn\\
&\qquad=C|v(t,\xi)|,
\end{align}
where we also used that the function $f$ is locally Lipschitz and the processes $u(\omega)$ and $X^a(\omega)$ are bounded. However, inequality \eqref{precarl} together with the fact that $v(0)=0$ implies that $v(t)=0$ for all $t\in[0,T]$ \cite[Theorem~3.0.4]{Vess}. This yields that
$$
0=v(T,\xi)=u(0,\xi)-a,\quad \xi\in\T^d.
$$
However it was assumed that $u(0)\notin A_c$. This contradiction finishes the proof for the case $x\in\C^2(\T^d)$.

In the general case, $x\in L_2(\T^d)$, we will use the Markov property of the solutions to \eqref{srde}, that is, Theorem~\ref{T:soltorus}(iii). We have for any $0<s<t$
$$
P_t(x,A_c)=\int_{L_2(\T^d)}P_s(x,dy)P_{t-s}(y,A_c)=
\int_{\C^2(\T^d)}P_s(x,dy)P_{t-s}(y,A_c)=\P(u^x(s)\in A_c),
$$
where the first identity is the Markov property, the second identity follows from the fact that $u^x(s)\in\C^2(\T^d)$ a.s. by Lemma~\ref{L:L2}, and the third identity follows from the fact that $P_{t-s}(y,A_c)=0$ if $y\in\C^2(\T^d)\setminus A_c$ established above. Since $s$ was arbitrary, we can  pass to the limit as $s\to0$. Using the fact that the process $u^x(\omega)$ is by definition continuous in $L_2(\T^d)$ and $x\notin A_c$, we get
$$
\P(u^x(s)\in A_c)\to 0,\quad \text{as $s\to0$},
$$
which completes the proof.

(iii).
It follows from Theorem~\ref{T:invmeasure} that SPDE \eqref{srde} has a unique invariant measure $\pi$ and that for any $x\in L_2(\T^d)$ one has $P_t(x,\cdot)\to \pi$ weakly as $t\to\infty$. Fix $x_0\in A_c$. Then, by part (i) of the proof and the Portmanteau theorem (see, e.g., \cite[Theorem III.1.1(II)]{Sch}), one has
$$
1=\limsup_{t\to\infty}P_t(x_0,A_c)\le \pi(A_c),
$$
where we also used the fact that the set $A_c$ is closed in $L_2(\T^d)$. This proves the statement of the lemma.
\end{proof}

\begin{proof}[Proof of Theorem~\ref{T:weareoptimal}]
Take $m=1$ and $\sigma_1\equiv1$.  Fix any  $x\in L_2(\T^d)\setminus A_c$. Then for any $t>0$ by Lemma~\ref{L:pochtifin}(ii) one has $P_t(x,A_c)=0$. On the other hand, by Lemma~\ref{L:pochtifin}(iii), $\pi(A_c)=1$. This implies the statement of the theorem.
\end{proof}

\end{document}